\documentclass[11pt]{article}

\usepackage{latexsym}
\usepackage{stmaryrd}
\usepackage{amssymb}
\usepackage{amsmath}
\usepackage{amscd}
\usepackage{amsthm}
\usepackage{graphicx}
\usepackage{indentfirst}
\usepackage[section]{placeins}
\usepackage{dsfont}
\usepackage{bm}
\usepackage{citesort}

\DeclareMathOperator{\essosc}{ess\,osc}

\def\Xint#1{\mathchoice
{\XXint\displaystyle\textstyle{#1}}%
{\XXint\textstyle\scriptstyle{#1}}%
{\XXint\scriptstyle\scriptscriptstyle{#1}}%
{\XXint\scriptscriptstyle\scriptscriptstyle{#1}}%
\!\int}
\def\XXint#1#2#3{{\setbox0=\hbox{$#1{#2#3}{\int}$ }
\vcenter{\hbox{$#2#3$ }}\kern-.6\wd0}}

\def\dashint{\Xint-}

\newtheorem{theorem}{Theorem}
\numberwithin{equation}{section}
\numberwithin{theorem}{section}
\newtheorem{lemma}[theorem]{Lemma}
\newtheorem{cor}[theorem]{Corollary}

\newtheorem{remark}[theorem]{Remark}

\newcommand{\e}{\epsilon}

\numberwithin{equation}{section}

\newcommand{\mb}{\mathbb}
\newcommand{\mc}{\mathcal}

\def\Re{\mathbb{R}}
\def\Int{\mathbb{Z}}
\def\Rat{\mathbb{Q}}

\title{\textsf{Homogenization for Rigid Suspensions with Random Velocity-Dependent Interfacial Forces}}
\author{Yuliya Gorb\thanks{
Department of Mathematics, University of Houston, Houston, TX,
77204, gorb@math.uh.edu}, Florian Maris\thanks{corresponding author,
Department of Mathematics, University of Houston, Houston, TX,
77204, fmaris@math.uh.edu}, Bogdan Vernescu\thanks{Department of Mathematical Sciences, Worcester Polytechnic Institute, 100 Institute Rd., Worcester, MA 01609, vernescu@wpi.edu}}
\date{\today}

\begin{document}

\maketitle \thispagestyle{empty}

 \begin{abstract}
\noindent We study suspensions of solid particles in a viscous incompressible fluid in the presence of highly oscillatory velocity-dependent surface forces. The flow at a small Reynolds number is modeled by the Stokes equations coupled with the motion of rigid particles arranged in a periodic array. The objective is to perform homogenization for the given suspension and obtain an equivalent description of a homogeneous (effective) medium, the macroscopic effect of the interfacial forces and the effective viscosity are determined using the analysis on a periodicity cell. In particular, the solutions $\bm{u}^\e_\omega$ to a family of problems corresponding to the size of microstructure $\e$ and describing suspensions of rigid particles with random surface forces imposed on the interface, converge $H^1$-- weakly as $\e \to 0$ a.s. to a solution of the so-called homogenized problem with constant coefficients. It is also shown that there is a corrector to a homogenized solution that yields a strong $H^1$-- convergence. 
The main technical construct is built upon the $\Gamma$-- convergence theory.

\end{abstract}


\section{Introduction}
Flows of incompressible fluids that carry rigid particles
also known as {\it particulate flows} are essential parts of many engineering
and environmental processes (i.e. particle sedimentation, fluidization and conveying) and 
are commonly encountered in many applications and fundamental fluid mechanics.
The complexity of mechanisms that govern fluid-particle and particle-particle
interactions makes the numerical simulation of these flows be one of the most challenging problems in computational
fluid dynamics. Many analytical, numerical, and experimental studies have been performed
during the past decades, however still much more research is needed for the fundamental understanding of these complex heterogeneous media.

For instance, the {\it rheological behavior} of particulate flows have been extensively studied
for over a century. Up to date the most investigated concentration regime of particles is the dilute one in which the hydrodynamic interactions between inclusions are normally ignored as the 
interparticle distance 
exceeds the range of the flows that appear due to the particle motion. Hence, it is possible in this case to
isolate the effect of particle-fluid interactions on the effective behavior of suspensions, as was done in the pioneering work
by Einstein \cite{einstein0} where the asymptotic study of the effective 
viscosity $\mu^*$ of a suspension of rigid particles in a small concentration regime was carried out.
A simple approximation to the 
viscosity $\mu^*$ in volume fraction $\phi$ of rigid neutrally buoyant spheres in the suspension had been formally derived there (see also \cite{einstein}):
\begin{equation} \label{E:einstein}
\mu^*\simeq\mu\left(1+\frac{5}{2}\phi\right) \quad \mbox{ as } \quad \phi\to 0~,
\end{equation}
where $\mu$ is the viscosity of the ambient fluid. The subsequent extension to the ellipsoidal particles was done in \cite{jeffery}.
Despite the seemingly simple linear relation \eqref{E:einstein} between $\mu^*$ and $\phi$, the rigorous justification
of the Einstein's formula \eqref{E:einstein} has been carried out over a century later,
in \cite{haines-maz} by constructing explicit upper and lower
bounds on the effective viscosity $\mu^*$.

A pairwise particle interaction in the dilute regime has been first taken into account in \cite{bat-green} to formally derive a
$O(\phi^2)$--correction to the Einstein's formula. Such an $O(\phi^2)$--correction in the case of a
periodic suspension of spherical fluid drops of viscosity $\eta\to\infty$ in a fluid of
viscosity $\mu$ was rigorously proven, also very recently, in \cite{ammari} using techniques different from \cite{haines-maz} based on the layer-potential approach previously developed by the authors.
For a more detailed description of asymptotic studies for the effective viscosity of dilute
suspensions we refer a reader to \cite{ammari} and references therein.

Another ``extreme'' regime of particle concentration is the dense packing of particles in a suspension 
when the typical interparticle distances are much smaller than their sizes. In such a regime
the effective rheological properties of suspensions exhibit singular behavior as the
characteristic distance between particles tends to zero (or equivalently, as $\phi$ tends to the maximal packing volume).
Such densely-packed particulate flows were also extensively studied both numerically and analytically
(see e.g. \cite{fa67,graham,nptf00,sb01} and references therein) over the past decades. For example, in \cite{fa67} a local formal asymptotic analysis based on a pair of closely spaced particles showed that
the effective viscosity, described by the viscous dissipation rate, exhibits a blow up of order $O(\delta^{-1})$,
where $\delta$ is the distance between the neighbors, whereas numerical study of \cite{sb01} revealed that
in some cases the blow up might be much weaker, e.g. of $O(|\ln \delta|)$. Such a discrepancy comes from the fact that
in the high packing regime the dynamics of particulate flows is driven by the long-range interactions between particles
and local asymptotics is not sufficient here, therefore
analysis of this concentration regime is quite challenging.
Luckily, the development of discrete network approximation techniques of
 \cite{bbp05,bgn05,bgn09}
 allowed to settle the disagreement between the results of \cite{fa67} and \cite{sb01}. Namely, it was shown that there are multiple blow ups of the effective viscosity as $\delta\to 0$ demonstrating examples of the cases
 when the stronger blow up degenerates so the weaker blow up becomes the leading one (see, in particular, \cite{bp07,bgn09}).

The third concentration regime when particles are of the same order as their
sizes is called the finite or moderate concentration regime.
Unlike in the regimes mentioned above, where the extreme properties in some sense
facilitated the corresponding analyses (i.e. negligible interactions between particles
in the dilute limit and the presence of strong lubrication forces between closely spaced particles
that contributed to the blow up of the effective viscosity in highly concentrated suspensions),
the case of finite 
concentrations is much harder to analyze. 
Such a regime was treated in \cite{nunan-keller} where a periodic array of spherical
particles was considered. Under the assumption
that all inclusions follow the shear motion of the fluid (formula (5) in \cite{nunan-keller})
it was shown that $\mu^*=O(\epsilon^{-1})$, where $\epsilon$ is the distance between periodically
distributed particles. This assumption is analogous to the well-known
Cauchy-Born hypothesis in solid state physics, which is known to be
not always true \cite{ft02}.
Also, a finite concentration regime was considered in
\cite{levy-sp} where an asymptotic expansion of the
effective viscosity was constructed assuming a periodic distribution of particles of arbitrary shape.
The formal two-scale homogenization was carried out there under the assumption that
the number of particles increases to infinity as their total volume remains constant.
We distinguish this work \cite{levy-sp} as being the closest in its goals
to the current paper as stated below.

We also point out that the definition of the effective properties of suspensions that was 
employed in \cite{levy-sp,nunan-keller}
is different from the aforementioned treatments of the dilute limit and dense packing regimes, where
the effective properties were determined from the equality of the viscous dissipation rates in the suspension and
the effective fluid. The latter definition of the effective rheological behavior is directly related to viscometric measurements
that necessarily include boundary conditions and applied forces.

\vspace{5pt}

The current paper focuses on particulate flows consisting
of a viscous incompressible fluid that carries rigid neutrally buoyant particles whose size
is comparable to the typical interparticle distances. 

When the surface forces are negligible, the effective behavior of suspensions is well understood. In
\cite{Bat70,KelRubMol67,HinLea75,levy-sp,BunVer} 
it was shown that suspensions behave like viscous fluids with a modified viscosity. 
This is also the case of emulsions studied formally in \cite{KelRubMol67} and in the framework of two-scale convergence on \cite{LipVer94}. For these systems surface tension effects are present in the effective viscosity. 

By contrast, other suspensions exhibit a different non-newtonian constitutive behavior.
In colloidal suspensions the interparticle forces, including van der Waals, electrostatic, steric,
and depletion forces, have an important role in the colloidal stability and in the suspension's
rheological behavior. In electrorheological fluids, surface electric forces, change the viscous
behavior in a viscoplastic or Bingham rigid-plastic behavior. In other materials electrokinetic
phenomena and the interaction with ionized particles play an important role, as in the transport
through natural clays, in the electrophysiology of cartilages and bones, in semiconductor
transport or in membrane or bulk ion-exchangers.

In all these examples, an important problem is the better understanding of the influence
of the highly oscillating surface forces in the constitutive behavior of the suspension. The
reversible constitutive change in electrorheological fluids, the swelling in clays and the lubricating properties
of connective tissues cannot be explained without taking into account interfacial phenomena
and field interactions.

The main goal of this paper is to derive and justify homogenization-type results of suspensions of rigid particles in the presence of highly oscillating random surface forces, dependent on the velocity field, and the microstructure in the limit of $\epsilon\to 0$,
where $\epsilon$ is the size of that microstructure. 
Here, we consider only noncolloidal suspensions in which hydrodynamic interactions are much stronger than the Brownian motions, 
hence, the latter are neglected. 

The relatively young homogenization theory specifically designed for analysis of highly
heterogeneous and microstructured materials (see \cite{bp89,blp78,jko94,sp80}) has been
an active area of research in recent decades. In this theory the effective material properties
of periodic materials are defined based on analysis on a periodicity cell, and then these properties
depend on the mechanics of constituents and geometry of the periodic array of particles but independent
of the external boundary conditions and applied forces. They are normally determined in the limit as the size of the microstructure $\e\to 0$, or equivalently, as the number of particles goes to infinity. Such an analysis is carried out in the present paper. 
In particular, the solutions $\bm{u}^\e_\omega$ to a family of problems corresponding to the size of microstructure $\e$ and describing suspensions of rigid particles with random surface forces imposed on the interface, converge $H^1$-- weakly as $\e \to 0$ a.s. to a solution of the so-called homogenized problem with constant coefficients. Convergence of the corresponding functionals describing the respecting problems is also shown. We also demonstrate that there is a corrector to the homogenized solution which yields a strong $H^1$-- convergence. 
 
Despite the fact that assumption about periodicity of a particle array is too restrictive for
suspensions of {\it moving} particles it has been extensively used in analysis of particulate flows (see e.g. \cite{fa67,graham,nunan-keller,levy-sp}). Since we attempt to develop
the homogenization theory for suspensions of rigid particles the periodicity assumption is imperative in our construct.
Also, we point out that in the particular case when the highly oscillatory random forces on the
boundaries of particles are absent we recover the formal asymptotics of \cite{levy-sp}.
The mathematical tools, that the main justification techniques are built upon, are based on $\Gamma$-- convergence theory \cite{dalint93}.

The paper is organized as follows.
In {\bf Section \ref{sec2}} we present fluid-particle problem setting including a
description of 
interface forces and the statement of existence and uniqueness of the solutions etc. The homogenization problem is
formulated in {\bf Section \ref{sec3}} and main results including cell problems, effective viscosity
derivation, convergence and corrector results are given in {\bf Section \ref{sec4}}. Conclusions are presented
in {\bf Section \ref{sec5}}. Proofs of auxiliary facts including existence and uniqueness of
solutions are given in {\bf Appendices}.

\vspace{5pt}

\noindent {\bf Acknowledgments.} Y. Gorb and F. Maris were supported by the NSF grant DMS-1016531. 

\section{Problem Formulation for the Particulate Flow}
\label{sec2}

Let $D$ be a bounded domain in $\Re^d$ with Lipschitz boundary. We assume that $D$ is split into two parts, a part $D_f$ in which we consider a viscous, incompressible fluid of viscosity $\mu$, and a part $D_r\subset\subset D$ that consists of a finite union of disjoint rigid particles, $D_r=\bigcup_{l=1}^N D_r^l$. Each rigid particle $D_r^l$ is assumed to be an open connected set with Lipschitz boundary, compactly included in $D$. In the following, by $\bm{n}$ we denote the unit normal on the boundary of $D_f$, directed outside the fluid region.

The following notations are also to be used:
$$\bm{a}\cdot\bm{b}=\sum_{i=1}^d a_ib_i,\ \ \ \mbox{ for all }\bm{a},\bm{b}\in\Re^d$$
$$\bm{A} :\bm{B}=\sum_{i,j=1}^d A_{ij}B_{ij},\ \ \ \mbox{ for all }\bm{A},\bm{B}\in\Re^{d\times d}$$
$$(\bm{a}\times\bm{b})_{ij}= a_ib_j-a_jb_i,\ \ \ \mbox{ for all }\bm{a},\bm{b}\in\Re^d, \ \ \mbox{ so }\bm{a}\times\bm{b}\in\Re^{d\times d}.$$

In the fluid region $D_f$ of the domain $D$, we introduce the velocity field $\bm{u}$ of the fluid and its pressure $p$. The strain rate tensor is denoted by $\bm{e}(\bm{u})$, defined componentwise by
\begin{equation}
\label{def.e}
e_{ij}(\bm{u})=\frac{1}{2}\left(\frac{\partial u_j}{\partial x_i}+\frac{\partial u_i}{\partial x_j}\right),\ \mbox{ for }0\leq i,j\leq d,
\end{equation}
and the stress tensor $\bm{\sigma}$ defined by
\begin{equation}
\label{def.sigma}
\bm{\sigma}=\bm{\sigma}(\bm{u},p)=-p\bm{I}+2\mu\bm{e}(\bm{u}).
\end{equation}
In $D_f$, the motion of the fluid is described by the stationary Stokes equation for an incompressible fluid 
\begin{equation}
\label{eqD_f}
\left\{
\begin{array}{rll}
-\operatorname{div}\bm{\sigma} = &\bm{f} \ &\mbox{in} \ D_f, \\
\operatorname{div}\bm{u}= &0 &\mbox{in}\ D_f, \\
\end{array}
\right.
\end{equation}
where $\bm{f}$ represents the body forces. On each rigid particle $D_r^l$, the velocity field $\bm{u}$ satisfies a rigid body motion
\begin{equation}
\label{eqD_r^l}
\bm{e}(\bm{u})=\bm{0} \ \mbox{ in }\ D_r^l, \ \mbox{for each } 0\leq l\leq N.
\end{equation}

We assume that for each $0\leq l\leq N$ there exists the superficial force $\bm{f}_l=\bm{f}_l(\bm{x},\bm{u})$ acting on the boundary of $D_r^l$ that depends on the position $\bm{x}$ as well as on the velocity field
$\bm{u}$, and we impose the balance of forces on every $D_r^l$ by

\begin{equation}
\label{balance.f}
-\int_{\partial D_r^l}\bm{\sigma}\bm{n}ds+\int_{D_r^l}\bm{f} d\bm{x}+\int_{\partial D_r^l}\bm{f}_l (\bm{x},\bm{u})ds=\bm{0},
\end{equation}
as well as the balance of torques
\begin{equation}
\label{balance.t}
-\int_{\partial D_r^l}\bm{\sigma}\bm{n}\times (\bm{x}-\bm{x}_l)ds+\int_{D_r^l}\bm{f}\times (\bm{x}-\bm{x}_l) d\bm{x}+\int_{\partial D_r^l}\bm{f}_l(\bm{x},\bm{u}) \times (\bm{x}
-\bm{x}_l)ds=\bm{0},
\end{equation}
where $\bm{x}_l$ is the center of mass of $D_r^l$. 
Equations (\ref{eqD_f}) -- (\ref{balance.t}) are supplied with boundary conditions on $\partial D$ that for simplicity are assumed to be no-slip conditions:
\begin{equation}
\label{noslip1}
\bm{u}=\bm{0} \ \mbox{ on }\ \partial D.
\end{equation}
The boundary conditions on $D_r^l$, for each $0\leq l\leq N$, also are chosen no-slip, thus the continuity of the velocity field across $\partial D_r^l$ is imposed:
\begin{equation}
\label{noslip2}
\llbracket\bm{u} \rrbracket=\bm{0} \ \mbox{ on }\ \partial D_r^l, \ \mbox{ for each }\ 0\leq l\leq N,
\end{equation}
where by $\llbracket\bm{u} \rrbracket=\bm{u}_f-\bm{u}_r$ we understand the jump of the velocity field $\bm{u}$, with $\bm{u}_f$ being the velocity of the fluid on the boundary of $D_r$ and $\bm{u}_r$ the velocity of the rigid part
on the boundary of $D_r$.

Collecting equations (\ref{eqD_f}) -- (\ref{noslip2}), we derive the system that we intend to study:
\begin{equation}
\label{system}
\left\{
\begin{array}{rll}
-\operatorname{div}\bm{\sigma} &=\bm{f} \ \mbox{in} \ D_f, \\
\operatorname{div}\bm{u}&=0 \ \mbox{in}\ D_f, \\
\bm{e}(\bm{u})&=\bm{0} \ \mbox{in}\ D_r, \\
\llbracket\bm{u} \rrbracket&=\bm{0} \ \mbox{on}\ \partial D_r, \\
\displaystyle\int_{\partial D_r^l}\bm{\sigma}\bm{n}ds&=\displaystyle\int_{D_r^l}\bm{f} d\bm{x}+\int_{\partial D_r^l}\bm{f}_l(\bm{x},\bm{u}) ds, \mbox{ for }0\leq l\leq N,\\
\displaystyle\int_{\partial D_r^l}\bm{\sigma}\bm{n}\times (\bm{x}-\bm{x}_l)ds&=\displaystyle\int_{D_r^l}\bm{f}\times (\bm{x}-\bm{x}_l) d\bm{x}\\
&+\displaystyle\int_{\partial D_r^l}\bm{f}_l(\bm{x},\bm{u}) \times (\bm{x}-\bm{x}_l)ds, \mbox{ for }0\leq l\leq N,\\
\bm{u}&=\bm{0} \ \mbox{ on }\ \partial D.
\end{array}
\right.
\end{equation}

For the system (\ref{system}), according to the no-slip boundary conditions imposed on $\partial D$ and on $\partial D_r$ we look for the solution $\bm{u}\in H_0^1(D)^d$ satisfying $\operatorname{div}\bm{u}=0$ in $D_f$
and $\bm{e}(\bm{u})=\bm{0}$ in $D_r$. Since the velocity field $\bm{u}$ inside each rigid particle satisfies both conditions $\operatorname{div}\bm{u}=0$ and $\bm{e}(\bm{u})=\bm{0}$, the solution $\bm{u}$ is a divergence free vector field from $H_0^1(D)^d$. To that end, we introduce the following subspaces of $H^1(D)^d$
\begin{equation}
\label{space.div}
V=\{\bm{v}\in H_0^1(D)^d\ |\ \operatorname{div}\bm{v}=0\},
\end{equation}
and
\begin{equation}
\label{space.u}
V_r=\{\bm{v}\in V\ |\ \bm{e}(\bm{v})=\bm{0} \mbox{ in } D_r\}.
\end{equation}

Therefore, we look for the solution $\bm{u}\in V_r$ and $p\in L^2(D_f)$ satisfying the stationary Stokes equation (\ref{eqD_f}) in $D_f$ in the weak sense and the equations (\ref{balance.f}) and (\ref{balance.t}) representing the
balance of forces and torques.

We assume that $\bm{f}\in L^2(D)^d$ and for every $0\leq l\leq N$ there exists a function $g_l(s,\bm{z})$, $g_l:\partial D_r^l\times\Re^d\to \Re$ such that for all $s\in\partial D_r^l$ and all $\bm{z}\in\Re^d$
\begin{equation}
\label{deff^l}
\bm{f}_l(s,\bm{z})=-\bm{\nabla}_{\bm{z}} g_l(s,\bm{z}),
\end{equation}
where by $\bm{\nabla}_{\bm{z}}$ we denote the gradient with respect to $\bm{z}$, and moreover $g_l:\partial D_r^l\times\Re^d\to \Re$ satisfies the following properties:

\vspace{3mm}
(G1) for every fixed $\bm{z}\in\Re^d$ the function $s\mapsto g_l(s,\bm{z})$ is measurable with respect to the surface measure on $\partial D_r^l$,

\vspace{3mm}
(G2) for every fixed $s\in\partial D_r^l$ the function $\bm{z}\mapsto g_l(s,\bm{z})$ is convex and Fr\'{e}chet differentiable at every point $\bm{z}\in\Re^d$, and the Fr\'{e}chet differential is continuous on
$\Re^d$,

\vspace{3mm}
(G3) $s\mapsto g_l(s,\bm{0})$ belongs to $L^1 (\partial D_r^l)$,

\vspace{3mm}
(G4) there exists $\gamma\in (0,1]$ and a function $a_l\in L^{\frac{2}{2-\gamma}}(\partial D_r^l)$, such that
\begin{equation}
\label{gl4}
g_l(s,\bm{z}_1)-g_l(s,\bm{z}_2)\leq C a_l(s)|\bm{z}_1-\bm{z}_2|^\gamma,
\end{equation}
for almost every $s\in\partial D_r^l$ and all $\bm{z}_1,\bm{z}_2\in\Re^d.$

\begin{remark}
\label{remark1}
The function $g_l:\partial D_r^l\times\Re^d\to \Re$, using (\ref{gl4}) satisfies
$$|g_l(s,\bm{z})|\leq |g_l(s,\bm{z})-g_l(s,\bm{0})| + |g_l(s,\bm{0})|\leq C a_l(s)|\bm{z}|^\gamma+ |g_l(s,\bm{0})|$$
which, after using H\"{o}lder's inequality becomes
$$|g_l(s,\bm{z})|\leq |g_l(s,\bm{0})|+ C\frac{2-\gamma}{2} a_l(s)^{\frac{2}{2-\gamma}}+C \frac{\gamma}{2}|\bm{z}|^2$$
so we have the following estimate
$$|g_l(s,\bm{z})|\leq b_l(s)+C|\bm{z}|^2,$$
for almost every $s\in\partial D_r^l$ and all $\bm{z}\in\Re^d$, where $b_l\in L^1(\partial D_r^l)$.
\end{remark}

\begin{remark}
\label{remark2}
The function $g_l:\partial D_r^l\times\Re^d\to \Re$ is a Carath\'{e}odory function, thus, for every vector field $\bm{u}\in H^1(D)^d$, the function,
$\bm{x}\mapsto g_l(\bm{x},\bm{u}(\bm{x}))$ is well defined on $\partial D_r^l$ and measurable with respect to the surface measure. Also, using {\bf Remark \ref{remark1}} the function $\bm{x}\mapsto g_l(\bm{x},\bm{u}(\bm{x}))$ is in 
$L^1(\partial D_r^l)$.
\end{remark}

\begin{remark}
\label{remark3}
From the convexity of $\bm{z}\mapsto g_l(s,\bm{z})$ we have the following subdifferential type inequality:
$$g_l(s,\bm{z}+\bm{w})\geq g_l(s,\bm{z})+\bm{w}\cdot\bm{\nabla}_{\bm{z}} g_l(s,\bm{z})$$
for almost every $s\in\partial D_r^l$ and all $\bm{z},\bm{w}\in\Re^d$, which implies that
$$|\bm{w}\cdot\bm{\nabla}_{\bm{z}} g_l(s,\bm{z})|\leq 2a_l(s)+C(|\bm{z}|^2+|\bm{w}|^2)$$
so taking all $\bm{w}$ such that $|\bm{w}|\leq |\bm{z}|+1$ we obtain that
$$|\bm{\nabla}_{\bm{z}} g_l(s,\bm{z})|\leq c_l(s)+C|\bm{z}|,$$
for almost every $s\in\partial D_r^l$ and all $\bm{z}\in\Re^d$, where $c_l\in L^2(\partial D_r^l)$. Also, in the same way as in {\bf Remark \ref{remark2}}, for every vector field $\bm{u}\in H^1(D)^d$,
the function $\bm{x}\mapsto\bm{\nabla}_{\bm{x}} g_l(\bm{x},\bm{u}(\bm{x}))$ is measurable and is in $L^2(\partial D_r^l)^d$.
\end{remark}
Next we give an equivalent variational formulation for the system (\ref{system}) and state the existence and uniqueness of a weak solution.

\begin{theorem}
\label{thvarfor}
The pair $\{\bm{u},p\}$, with $\bm{u}\in V_r$ and $p\in L^2(D_f)$ is a weak solution for the system (\ref{system}) if and only if for every $\bm{\phi}\in H_0^1(D)^d$ such that
$\bm{e}(\bm{\phi})=\bm{0}$ in $D_r$ one has
\begin{equation}
\label{eqvarfor}
\int_D 2\mu\bm{e}(\bm{u}):\bm{e}(\bm{\phi})d\bm{x}-\int_{D_f} p \operatorname{div}\bm{\phi}d\bm{x}=\int_D\bm{f}\cdot\bm{\phi}d\bm{x}+
\sum_{l=1}^{N}\int_{\partial D_r^l}\bm{f}_l(\bm{x},\bm{u}) \cdot\bm{\phi}ds.
\end{equation}
\end{theorem}

The proofs of {\bf Theorem \ref{thvarfor}} and the following {\bf Theorems \ref{thminfor}} and {\bf \ref{thexun}} are given in the Appendix.

\begin{theorem}
\label{thminfor}
If $\{\bm{u},p\}$, with $\bm{u}\in V_r$ and $p\in L^2(D_f)$ is a weak solution for the system (\ref{system}) then $\bm{u}$ is the unique minimizer of the functional $E_r:H_0^1(D)^d\to\Re\cup\{+\infty\}$, defined by
$$E_r(\bm{v})=\int_D \mu\bm{e}(\bm{v}):\bm{e}(\bm{v})d\bm{x}-\int_D\bm{f}\cdot\bm{v}d\bm{x}+\sum_{l=1}^{N}\int_{\partial D_r^l}g_l(\bm{x},\bm{v})ds+I_{V_r},$$
where $I_S$ represents the indicator function of the set $S$, defined by
\begin{equation}
\nonumber
I_S(s)=
\left\{
\begin{array}{rl}
0 &\mbox{ if } s\in S\\
+\infty &\mbox{ if } s\not\in S.
\end{array}
\right.
\end{equation}
\end{theorem}

\begin{theorem}
\label{thexun}
There exists a weak solution $\{\bm{u},p\}$, with $\bm{u}\in V_r$ and $p\in L^2(D_f)$ of the system (\ref{system}), with $\bm{u}$ being unique, and $p$ unique up to a constant.
\end{theorem}

\section{Formulation of the Homogenization Problem}
\label{sec3}
In this section we formulate the homogenization problem. We focus on a suspension of rigid particles in a viscous fluid, periodically distributed, where the size of the particles is of the same order as the size of the period.

Let $D$ be a bounded domain in $\Re^d$ with Lipschitz boundary and let $Y=\left(-\dfrac{1}{2},\dfrac{1}{2}\right)^d$ be the unit cube in $\Re^d$. For every $\e > 0$, let $N^\e$ be the set of all points $k\in\Int^d$ such that $\e(k+Y)$ is strictly included in $D$ and denote by $|N^\e|$ the total number of them. Let $T$ be the closure of an open connected set with Lipschitz boundary, compactly included in $Y$. For every $\e > 0$ and $k\in N^\e$ we consider the set $T^\e_k\subset\subset \e(k+Y)$, where $T^\e_k=\e(k+T)$. The set $T^\e_k$ represents one rigid particle, suspended in the fluid occupying the domain $D$.

We now define the following subsets of $D$:
$$D^\e_r=\displaystyle\bigcup_{k\in N^\e} T^\e_k\ ,\ \ \ \ D^\e_f=D\setminus D^\e_r.$$ 
As in {\bf Section \ref{sec2}}, $D^\e_r$ represents the union of the rigid suspensions, and $D^\e_f$ the part of the domain $D$ filled with a viscous incompressible fluid of viscosity $\mu$. Let $\bm{n}$ be the unit normal on the boundary of $D^\e_f$ that points out of the fluid region.

Now we describe the random superficial forces acting on the boundaries of all $T^\e_k$, for $k\in N^\e$. Let $(\Omega,\mc{F},P)$ be a probability space and assume there exists a function $g:\partial T\times \Re^d\times \Omega\to \Re$, extended Y-periodically in the first variable to a function defined on $\bigcup_{k\in\Int^d}(k+\partial T)\times \Re^d\times \Omega$, with the following properties:

\vspace{3mm}
(H1) for every fixed $\bm{z}\in\Re^d$ and for almost every $\omega\in\Omega$ the function $s\mapsto g(s,\bm{z},\omega)$ is measurable with respect to the surface measure on $\partial T$,

\vspace{3mm}
(H2) for almost every $s\in\partial T$ and for almost every $\omega\in\Omega$ the function $\bm{z}\mapsto g(s,\bm{z},\omega)$ is convex and Fr\'{e}chet differentiable at every point $\bm{z}\in\Re^d$, and the Fr\'{e}chet differential is continuous on
$\Re^d$,

\vspace{3mm}
(H3) for every fixed $\bm{z}\in\Re^d$ and for almost every $s\in\partial T$ the function $\omega\mapsto g(s,\bm{z},\omega)$ is measurable with respect to $P$,

\vspace{3mm}
(H4) $s\mapsto g(s,\bm{0},\omega)$ belongs to $L^1 (\partial T \times \Omega)$,

\vspace{3mm}
(H5) there exists $\gamma\in (0,1]$ and a function $a\in L^{\frac{2}{2-\gamma}}(\partial T\times \Omega)$, such that
\begin{equation}
\label{h5}
g(s,\bm{z}_1,\omega)-g(s,\bm{z}_2,\omega)\leq C a(s,\omega)|\bm{z}_1-\bm{z}_2|^\gamma,
\end{equation}
for almost every $s\in\partial T$, a.s. $\omega\in\Omega$, and all $\bm{z}_1,\bm{z}_2\in\Re^d.$

\vspace{3mm}
On the probability space $(\Omega,\mc{F},P)$, we consider a $d$-dimensional dynamical system $\tau$, which is a family of mappings $(\tau_k)_{k\in\Int^d}$ on $\Omega$ that satisfy the following properties:

\vspace{3mm}
(T1) Group property, i.e. $\tau_0$ is the identity and
$$\tau_k\circ\tau_l=\tau_{k+l} \ \ \mbox{ for all } k,l\in\Int^d$$

(T2) Invariance, i.e. the mappings $(\tau_k)_{k\in\Int^d}$ are measurable and measure preserving, which means that
$$P(\tau_k^{-1}{B})=P(B) \ \ \mbox{ for all } k,l\in\Int^d \ \mbox{and} \ B\in\mc{F}$$

If the system $\tau$ satisfies the additional ergodicity property:

\vspace{3mm}
(T3)
$$\mbox{if }\tau_k{B}=B \mbox{ for all } k\in\Int^d \mbox{ then } B\in\{\emptyset,\Omega\}$$
then the system is called ergodic.

\vspace{3mm}
For every ${k\in\Int^d}$, and to any measurable function $f$ in $(\Omega,\mc{F},P)$ we associate the function $\tau_k f=f\circ\tau_k$ which has the same distribution as $f$. Therefore, the dynamical system induces a $d$-parameter group of isometries on $L^p(\Omega)$ for $1\leq p\leq\infty$.

A set $B\in\mc{F}$ is called $\tau$-invariant if $P(\tau_kB\Delta B)=0$ for every ${k\in\Int^d}$, where $\Delta$ denotes symmetric difference. A measurable function $f$ is called $\tau$-invariant if for every ${k\in\Int^d}$: $\tau_k f=f$ a.s.
It can be easily shown that the $\tau$-invariant sets form a sub $\sigma$-algebra of $\mc{F}$, denoted by $\mc{I}$, and a function is $\tau$-invariant if and only if it is measurable with respect to $\mc{I}$.
Then, the ergodicity property is equivalent to any of the following:
$$ B \ \tau\mbox{-invariant }\mbox{ implies } P(B)\in\{0,1\},$$
$$ f \ \tau\mbox{-invariant }\mbox{ implies } f\equiv\mbox{constant} \ \ a.s.$$
Given $\tau$ a $d$-dimensional dynamical system, to any function $f$ from $L^1(\Omega)$ we associate an additive process on finite subsets $F$ from $\Int^d$ with values in $L^1(\Omega)$:
$$S(F,f)(\omega)=\sum_{k\in F}\tau_k f(\omega)$$
and for every $F$ we denote by $A(F,f)$ the following average of $f$
$$A(F,f)(\omega)=\displaystyle\frac{1}{|F|}S(F,f)(\omega).$$
We now use the pointwise Ergodic Theorem (\cite{krengel1}, Ch. 6 or \cite{dunsch1} Ch. 8), which states that the sequence of averages $A(nQ,f)$ will converge pointwise to the expected value of the function $f$ with respect to the $\sigma$-- algebra $\mc{I}$, $\mb{E}(f|\mc{I})$ a.s. $\omega\in\Omega$, where $Q$ is a nontrivial cube in $\Int^d$, i.e. there exists $u,v\in\Int^d$ with $u_i< v_i$ for every $i\in \{1,\ldots,d\}$ and $$Q=[u,v]=\{k\in\Int^d\ |\ u\leq k \leq v\}.$$ If $\tau$ is ergodic, then $A(nQ,f)$ converges pointwise to $\displaystyle\int_\Omega f dP$ a.s. $\omega\in\Omega$.

We present now an equivalent formulation of the pointwise Ergodic Theorem (see \cite{marver2} for the proof of equivalence) that will be applied throughout this paper. If $U\in\Re^d$, open, bounded with Lipschitz boundary and $\e>0$, we use the notation
$$\dfrac{U}{\e}=\{\bm{x}\in\Re^d\ |\ \e\bm{x}\in U\}.$$
Then, if $f\in L^1(\Omega)$ and $\tau$ is an ergodic dynamical system acting on $\Omega$, the sequence of averages $A\left(\Int^d\cap \dfrac{U}{\e},f\right)$ will converge pointwise to $\displaystyle\int_\Omega f dP$ a.s. $\omega\in\Omega$.

Define the functions $g^\e_k:\partial T^\e_k\times\Re^d\times\Omega\to \Re$ by
\begin{equation}
\label{defg^e_k}
g_k^\e(s,\bm{z},\omega)=\e g\left(\frac{s}{\e},\bm{z},\tau_k\omega\right),
\end{equation}
for every $k\in N^\e$, where $\e$ is the size of the periodic structure, the function $g$ satisfies (H1) -- (H5) and the family $(\tau_k)_{k\in\Int^d}$ satisfies (T1) -- (T3). We notice now that a.s. $\omega\in\Omega$ the function $(s,\bm{z})\mapsto g_k^\e(s,\bm{z},\omega)$ satisfies (G1) -- (G4).

If $\bm{f}^\e\in L^2(D)$ is the body force and if the random superficial forces acting on the boundaries of $T^\e_k$ are given for every $\omega\in\Omega$ by
\begin{equation}
\label{deff^e_k}
(\bm{f}_k^\e)_\omega=-\bm{\nabla}_{\bm{z}}g_k^\e(s,\bm{z},\tau_k\omega),
\end{equation}
then for every $\omega\in\Omega$, the system that we intend to study is the equivalent to the system (\ref{system})
\begin{equation}
\label{systemhe}
\left\{
\begin{array}{rll}
-\operatorname{div}\bm{\sigma}^\e_\omega &=\bm{f}^\e \ \mbox{in} \ D_f^\e, \\
\operatorname{div}\bm{u}^\e_\omega &=0 \ \mbox{in}\ D_f^\e, \\
\bm{e}(\bm{u}^\e_\omega)&=\bm{0} \ \mbox{in}\ D_r^\e, \\
\llbracket\bm{u}^\e_\omega \rrbracket&=\bm{0} \ \mbox{on}\ \partial D_r^\e, \\
\displaystyle\int_{\partial T^\e_k}\bm{\sigma}^\e_\omega\bm{n}ds&=\displaystyle\int_{T^\e_k}\bm{f}^\e d\bm{x}+\int_{\partial T^\e_k}\bm{f}_k^\e(\bm{x},\bm{u}^\e,\tau_k\omega) ds, \mbox{ for }\ k\in N^\e, \\
\displaystyle\int_{\partial T^\e_k}\bm{\sigma}^\e_\omega\bm{n}\times (\bm{x}-\bm{x}^\e_k)ds&=\displaystyle\int_{T^\e_k}\bm{f}^\e\times (\bm{x}-\bm{x}^\e_k) d\bm{x}\\
&+\displaystyle\int_{\partial T^\e_k}\bm{f}_k^\e(\bm{x},\bm{u},\tau_k\omega) \times (\bm{x}-\bm{x}^\e_k)ds, \mbox{ for }\ k\in N^\e, \\
\bm{u}^\e_\omega &=\bm{0} \ \mbox{ on }\ \partial D.
\end{array}
\right.
\end{equation}
where, for every $\omega\in\Omega$, $\bm{u}^\e_\omega $ is the velocity of the fluid, $p^\e_\omega $ is the pressure and $\bm{\sigma}^\e_\omega =2\mu\bm{e}(\bm{u}^\e_\omega)-p^\e _\omega I$ is the stress tensor associated to $\bm{u}^\e_\omega $ and $p^\e_\omega $. In (\ref{systemhe}), $\bm{x}^\e_k$ is the center of mass of $T^\e_k$, $\bm{x}^\e_k=\e k+ \e\bm{x}_0$, where $\bm{x}_0$ is the center of mass of $T$.

We define the space $V^\e$, equivalent to the space $V_r$ defined in (\ref{space.u})
\begin{equation}
\label{space.ue}
V^\e=\{\bm{v}\in H_0^1(D)^d\ |\ \operatorname{div}\bm{v}=0 \mbox{ and }\bm{e}(\bm{v})=\bm{0} \mbox{ in } D_r^\e\},
\end{equation}
and we look for a weak solution of the system (\ref{systemhe}), $\bm{u}^\e_\omega \in V^\e$ and $p^\e_\omega \in L^2(D_f^\e)$.

We know from {\bf Theorem \ref{thvarfor}} that a weak solution of the system (\ref{systemhe}) satisfies the following equivalent variational formulation
\begin{equation}
\label{eqvarforhe}
\int_D 2\mu\bm{e}(\bm{u}^\e_\omega):\bm{e}(\bm{\phi})d\bm{x}-\int_{D_f^\e} p^\e_\omega \operatorname{div}\bm{\phi}d\bm{x}=\int_D\bm{f}^\e\cdot\bm{\phi}d\bm{x}+\sum_{k\in N^\e}\int_{\partial T^\e_k}-\bm{\nabla}_{\bm{u}}\e g\left(\frac{s}{\e},\bm{u},\tau_k\omega\right) \cdot\bm{\phi}ds,
\end{equation}
for every $\bm{\phi}\in H_0^1(D)^d$ such that $\bm{e}(\bm{\phi})=\bm{0}$ in $D^\e_r$. We also have, from {\bf Theorem \ref{thexun}} that the solution exists and is unique in $V^\e\times L^2(D_f^\e)$ and from {\bf Theorem \ref{thminfor}} that $\bm{u}^\e_\omega$ is the minimizer of the functional $E^\e_\omega:H_0^1(D)^d\to\Re\cup\{+\infty\}$, defined by
\begin{equation}
\label{defE^e_omega}
E^\e_\omega(\bm{v})=\int_D \mu\bm{e}(\bm{v}):\bm{e}(\bm{v})d\bm{x}-\int_D\bm{f}^\e\cdot\bm{v}d\bm{x}+\sum_{k\in N^\e}\int_{\partial T^\e_k}\e g\left(\frac{s}{\e},\bm{v},\tau_k\omega\right)ds+I_{V^\e}(\bm{v}).
\end{equation}

In the next chapter we study the limiting behaviour of the system (\ref{systemhe}) as $\e\to 0$ for the suspension of rigid particles in an incompressible viscous flow and determine the effective properties of the suspension. We show that the solutions of the system (\ref{systemhe}) converges when $\e\to 0$, to a solution of a system that we will describe later. We take advantage of the minimization property of the solution $\bm{u}^\e_\omega$, and justify the obtained asymptotics using the $\Gamma$-- convergence method, introduced by De Giorgi, that implies the convergence of minimizers.

\section{Homogenization Results}
\label{sec4}
\subsection{Cell Problems}
In order to determine the $\Gamma$-- limit of functionals $E^\e_\omega$ we need first to define the so-called ``cell problems'' and their solutions. Let $\mc{S}$ be the subset of $\Re^{d\times d}$ consisting of symmetric matrices. For any $\bm{A}\in\mc{S}$, consider the following minimization problem:
\begin{equation}
\label{solcell1}
\min_{\bm{v}\in K_{\bm{A}}} \int_Y \mu [\bm{A}\mathds{1}_T-\bm{e}(\bm{v})]:[\bm{A}\mathds{1}_T-\bm{e}(\bm{v})] d\bm{y},
\end{equation}
where
\begin{equation}
\label{defK_A}
K_{\bm{A}}=\{\bm{v}\in H_{per}^1(Y)^d / \Re\ |\ \bm{e}(\bm{v})=\bm{A}\mbox{ in } T,\operatorname{div}\bm{v}=0\mbox{ in } Y\setminus T \},
\end{equation}
and $\mathds{1}_T$ represents the characteristic function of the set $T$, defined by
\begin{equation}
\nonumber
\mathds{1}_T(\bm{y})=
\left\{
\begin{array}{rl}
1 &\mbox{ if }\bm{y}\in T\\
0 &\mbox{ if }\bm{y}\in Y\setminus T.
\end{array}
\right.
\end{equation}

The choice of the space $H_{per}^1(Y)^d / \Re$, which we identify from now on with the subspace of $H_{per}^1(Y)^d$ consisting of zero mean vector fields, is justified by the fact that the functional we are trying to minimize $\bm{v}\to \displaystyle\int_Y \mu [\bm{A}\mathds{1}_T-\bm{e}(\bm{v})]:[\bm{A}\mathds{1}_T-\bm{e}(\bm{v})] d\bm{y}$ does not change if we add to any $\bm{v}\in H^1_{per}(Y)^d$ a constant, thus on $H^1_{per}(Y)^d$ the functional is not coercive and also we cannot have uniqueness of the solution.

A solution for the minimization problem (\ref{solcell1}) thus exists because the functional $\bm{v}\to \displaystyle\int_Y \mu [\bm{A}\mathds{1}_T-\bm{e}(\bm{v})]:[\bm{A}\mathds{1}_T-\bm{e}(\bm{v})] d\bm{y}$ is convex, continuous in the strong topology of $H^1(Y)^d / \Re$ and coercive, and the set $K_{\bm{A}}$ is closed and convex.
The minimizer $\bm{\chi}_{\bm{A}}\in K_{\bm{A}}$ may be also characterized by
\begin{equation}\nonumber
\int_Y 2\mu [\bm{A}\mathds{1}_T-\bm{e}(\bm{\chi}_{\bm{A}})]:[\bm{e}(\bm{\chi}_{\bm{A}})-\bm{e}(\bm{v})] d\bm{y} =0,
\end{equation}
for every $\bm{v}\in K_{\bm{A}}$, which is equivalent to
\begin{equation}
\label{solcell3}
\int_Y 2\mu [\bm{A}\mathds{1}_T-\bm{e}(\bm{\chi}_{\bm{A}})]:\bm{e}(\bm{\phi}) d\bm{y} =0,
\end{equation}
for every $\bm{\phi}\in K_{\bm{0}}$. We easily see from here that $\bm{\chi}_{\bm{A}}\in K_{\bm{A}}$ is unique.

The formulation (\ref{solcell3}) also implies that $\bm{\chi}_{\bm{A}}$ is the solution of the following Stokes system (see also \cite{levy-sp} for the formulation of the cell problem):
\begin{equation}
\label{solcell4}
\left\{
\begin{array}{rll}
-\Delta\bm{\chi}_{\bm{A}} +\bm{\nabla} \eta_{\bm{A}} &=\bm{0} \ \mbox{in} \ Y\setminus T, \\
\operatorname{div}\bm{\chi}_{\bm{A}}&=0 \ \mbox{in}\ Y\setminus T, \\

\bm{e}(\bm{\chi}_{\bm{A}})&=\bm{A} \ \mbox{in}\ T, \\
\bm{\chi}_{\bm{A}}&\in H^1_{per}(Y)^d / \Re, \\
\end{array}
\right.
\end{equation}
with $\bm{\chi}_{\bm{A}}\in K_{\bm{A}}$ and $\eta_{\bm{A}}\in L^2(Y\setminus T) / \Re$ are unique.

If we extend $\mathds{1}_T$ and $\bm{\chi}_{\bm{A}}$ by periodicity to the whole $\Re^d$, and keep for simplicity the same notations, we obtain that $\bm{\chi}_{\bm{A}}\in H^1_{loc}(\Re^d)^d$ as a consequence of periodicity. Moreover, the following result holds:
\begin{theorem}
\label{thcellglob}
For every divergence free vector field $\bm{\phi}\in H^1(\Re^d)^d$ with bounded support, having the property that $\bm{e}(\bm{\phi})_{|k+T}=\bm{0}$, $\forall k\in\Int^d$ one has
\begin{equation}
\label{cellglob1}
\int_{\Re^d} 2\mu [\bm{A}\mathds{1}_T-\bm{e}(\bm{\chi}_{\bm{A}})]:\bm{e}(\bm{\phi}) d\bm{x} =0.
\end{equation}
\end{theorem}

\begin{proof}
Here we follow the idea of \cite{ciodon}, Th. 4.28, and cover the support of $\bm{\phi}$ with finitely many translated cells of $Y$, $(Y_i)_{i=1}^m$, chosen in such a way that $\partial{(k+T)}$ and $\partial Y_i$ are disjoint for every $k\in\Int^d$ and every $1\leq i\leq m$. From the choice of the sets $(Y_i)_{i=1}^m$, it is obvious that for each $1\leq i \leq m$ there exists a unique $k_i\in\Int^d$ such that $k_i+T \subset\subset Y_i$. So, we can associate a smooth partition of unity $(\theta_i)_{i=1}^m$ to the covering $(Y_i)_{i=1}^m$ with the properties
\begin{equation}\nonumber
\left\{
\begin{array}{l}
\theta_i\in C^\infty_0(Y_i),\ 0\leq \theta_i\leq 1,\ \theta_i\equiv 1\mbox{ on a neighborhood of } k_i+T,\ \forall i:\ 1\leq i\leq m,\\
\\
\displaystyle\sum_{i=1}^m \theta_i = 1 \mbox{ on a neighborhood of } \mbox{ supp }\bm{\phi}.
\end{array}
\right.
\end{equation}
If $(\theta_i\bm{\phi})^\#$ represents the vector field defined in $Y$ after we extend $Y$-- periodically $\theta_i\bm{\phi}$, then $(\theta_i\bm{\phi})^\#$ is $Y$-- periodic and $\bm{e}((\theta_i\bm{\phi})^\#)\equiv\bm{0}$ on $T$, so
\begin{equation}\nonumber
\begin{split}
\int_{\Re^d} 2\mu [\bm{A}\mathds{1}_T-\bm{e}(\bm{\chi}_{\bm{A}})]:\bm{e}(\bm{\phi}) d\bm{x}&=\sum_{i=1}^m\int_{\Re^d} 2\mu [\bm{A}\mathds{1}_T-\bm{e}(\bm{\chi}_{\bm{A}})]:\bm{e}(\theta_i\bm{\phi}) d\bm{x}\\
&=\sum_{i=1}^m\int_{Y_i} 2\mu [\bm{A}\mathds{1}_T-\bm{e}(\bm{\chi}_{\bm{A}})]:\bm{e}(\theta_i\bm{\phi}) d\bm{x}\\
&=\sum_{i=1}^m\int_{Y} 2\mu [\bm{A}\mathds{1}_T-\bm{e}(\bm{\chi}_{\bm{A}})]:\bm{e}((\theta_i\bm{\phi})^\#) d\bm{y},\\
\end{split}
\end{equation}
and using the formulation (\ref{solcell4}) and the properties of the functions $(\theta_i)_{i=1}^m$ we obtain that
\begin{equation}\nonumber
\begin{split}
\int_{\Re^d} 2\mu [\bm{A}\mathds{1}_T-\bm{e}(\bm{\chi}_{\bm{A}})]:\bm{e}(\bm{\phi}) d\bm{x}&=\sum_{i=1}^m\int_{Y\setminus T} -\eta_{\bm{A}} \operatorname{div}(\theta_i\bm{\phi})^\# d\bm{y}\\
&=\int_{Y} -\eta_{\bm{A}} \operatorname{div}(\sum_{i=1}^m\theta_i\bm{\phi})^\# d\bm{y}=0.
\end{split}
\end{equation}
\end{proof}
We have the following consequence of {\bf Theorem \ref{thcellglob}}:
\begin{cor}
\label{cor1}
For any $\bm{A}\in\mc{S}$, one has $\bm{\chi}_{\bm{A}}\in H^2(Y)^d$ and $\eta_{\bm{A}}\in H^1(Y\setminus T)$ and the following estimates hold:
$$||\bm{\chi}_{\bm{A}}||_{H^2(Y)^d}\leq C ||\bm{A}||,~~\quad ||\eta_{\bm{A}}||_{H^1(Y\setminus T)}\leq C ||\bm{A}||,$$
where the constant $C$ is independent of $\bm{A}\in\mc{S}$.
\end{cor}
\begin{proof}
The first part follows from the regularity theorem (\cite{temam}, Proposition 2.2), all we have to show is that the boundary values of $\bm{\chi}_{\bm{A}}$ on $\partial T$ and $\partial Y$ are smooth. But the condition $\bm{e}(\bm{\chi}_{\bm{A}})=\bm{0}$ on $T$ implies that $\bm{\chi}_{\bm{A}}\in C^\infty (T)^d$ and {\bf Theorem \ref{thcellglob}} tells us that outside the sets $k+T$, $\bm{\chi}_{\bm{A}}$ is also infinitely differentiable being the solution of the Stokes system with $\bm{0}$ body force.

To obtain the estimates we notice first the linearity of the map $\bm{A}\mapsto\bm{\chi}_{\bm{A}}$ and then the fact that $\bm{A}\mapsto||\bm{\chi}_{\bm{A}}||_{H^2(Y)^d}$ is a norm on $\Re^{d\times d}$.
\end{proof}
Now define the bilinear form $\mc{C}$ on $\mc{S}$ as follows:
\begin{equation}
\label{defC}
\mc{C}[\bm{A},\bm{B}]= \int_Y \mu\bm{e}(\bm{\chi}_{\bm{A}}):\bm{e}(\bm{\chi}_{\bm{B}}) d\bm{y}.
\end{equation}
\subsection{Auxiliary Lemmas}
In this subsection we state and prove several auxiliary results needed to show the convergence results from {\bf Subsection \ref{ssec4.3}}.

In the $\Gamma$-- convergence proof we need first to approximate any vector field from $V$ with a sequence of vector fields from $V^\e$ such that we have also convergence for the functionals. {\bf Lemmas \ref{lemmaextension} -- \ref{lemmasuf}} are preliminary results needed to perform the construction. Our strategy is to use the solution of the local problems and construct first a sequence in $H_0^1(D)^d$ with all the required properties to be satisfied up to some small error, and then using a diagonalization argument in {\bf Lemma \ref{lemmasuf}} to obtain the recovery sequence. In {\bf Lemma \ref{lemmasuf}} we need also to estimate how well one can approximate a vector field from $H^1(D)^d$ with one from $V^\e$, and the result is given in {\bf Lemma \ref{lemmaconstruction}}. The periodic arrangement of the rigid particles is very important here and we take advantage of it using {\bf Lemma \ref{lemmaextension}} which is an extension type of result in the unit cell.
\begin{lemma}
\label{lemmaextension}
Let $\bm{v}\in H^1(T)^d$. Then, there exists a vector field $\tilde{\bm{v}}\in H_0^1(Y)^d$, such that $\bm{e}(\bm{v})=\bm{e}(\tilde{\bm{v}})$ in $T$ and $||\tilde{\bm{v}}||_{H_0^1(Y)^d}\leq C ||\bm{e}(\bm{v})||_{L^2(T)^{d\times d}}$. If $\operatorname{div}\bm{v}\equiv 0$ in $T$, then $\tilde{\bm{v}}$ may be also chosen so that $\operatorname{div}\bm{v}\equiv 0$ in $Y$. 
\end{lemma}

\begin{proof}
Let $V_1$ be the closed subspace of $H^1(T)^d$ that consists of vector fields $\bm{v}_1$ such that $\bm{e}(\bm{v}_1)=\bm{0}$ almost everywhere in $T$, and $V_2$ the complement of $V_1$ in $H^1(T)^d$. On $V_2$, we consider the norms $||\cdot||_1$ and $||\cdot||_2$ given by:
$$||\bm{v}_2||_1=||\bm{v}_2||_{H^1(T)^d}$$
and
$$||\bm{v}_2||_2=\left(\int_T\bm{e}(\bm{v}_2):\bm{e}(\bm{v}_2)d\bm{y}\right)^{1/2}.$$
The norm $||\cdot||_1$ is the norm induced and $||\cdot||_2$ is a norm as a consequence of the definition of $V_2$. Obviously we have that $||\bm{v}_2||_2\leq||\bm{v}_2||_1$ for every $\bm{v}_2\in V_2$, so as a consequence of the Open Mapping Theorem, there exists a constant $C$ such that for every $\bm{v}_2\in V_2$ we have
$$||\bm{v}_2||_1\leq C ||\bm{v}_2||_2.$$
Consider $\bm{v}\in H^1(T)^d$, and let $\bm{v}=\bm{v}_1+\bm{v}_2$ with $\bm{v}_i\in V_i$ for $i=1,2$. Let $\tilde{\bm{v}}_2$ be an extension of $\bm{v}_2$, $\tilde{\bm{v}}_2\in H_0^1(Y)^d$, and $||\tilde{\bm{v}}_2||_{H_0^1(Y)^d}\leq C ||\bm{v}_2||_1$. We define then $\tilde{\bm{v}}$ to be $\tilde{\bm{v}}_2$, so in $T$ we have
$$\bm{e}(\tilde{\bm{v}})=\bm{e}(\tilde{\bm{v}}_2)=\bm{e}(\bm{v}_2)=\bm{e}(\bm{v}_1+\bm{v}_2)=\bm{e}(\bm{v})$$
and
$$||\tilde{\bm{v}}||_{H_0^1(Y)^d}\leq C ||\bm{v}_2||_1\leq C||\bm{v}_2||_2=C\left(\int_T\bm{e}(\bm{v}):\bm{e}(\bm{v})d\bm{y}\right)^{1/2}.$$
Now, if $\operatorname{div}\bm{v}\equiv 0$ in $T$, let $\tilde{\bm{v}}_2$ be the extension obtained as in the previous case, and consider the weak solution of the following Stokes system:
\begin{equation}\nonumber
\left\{
\begin{array}{rll}
-\Delta\bm{w}+\bm{\nabla} p=&\bm{0} \ & \mbox{ in } Y\setminus T\\
\operatorname{div}\bm{w}=&\operatorname{div}\tilde{\bm{v}}_2 \ & \mbox{ in } Y\setminus T\\
\bm{w}=&\bm{0} \ & \mbox{ on } \partial Y\cup\partial T
\end{array}
\right.
\end{equation}
that satisfies the estimate
$$||\bm{w}||_{H_0^1(Y\setminus T)^d}\leq C ||\operatorname{div}\tilde{\bm{v}}_2||_{L^2(Y\setminus T)^d}.$$
We take now $\tilde{\bm{v}}$ to be $\tilde{\bm{v}}_2-\bm{w}$ and $\tilde{\bm{v}}$ obviously satisfies the three properties.
\end{proof}

\begin{lemma}
\label{lemmaconstruction}
Let $\bm{u}\in H_0^1(D)^d$. Then, there exists a vector field $ {\bm{u'}}\in V^\e$ with the property that $||\bm{u}- {\bm{u'}}||_{H_0^1(D)^d}\leq C\left(||\operatorname{div}\bm{u}||_{L^2(D)}+||\bm{e}(\bm{u})||_{L^2(D_r^\e)^{d\times d}}\right)$,
where the constant $C$ is independent of $\e$.
\end{lemma}
\begin{proof}
We first construct $\bm{u}_1\in H_0^1(D)^d$ as the solution for the Stokes system
\begin{equation}\nonumber
\left\{
\begin{array}{rll}
-\Delta\bm{u}_1+\bm{\nabla} p_1=&\bm{0} \ & \mbox{ in } D\\
\operatorname{div}\bm{u}_1=&\operatorname{div}\bm{u} \ & \mbox{ in } D\\
\bm{u}_1=&\bm{0} \ & \mbox{ on } \partial D
\end{array}
\right.
\end{equation}
The solution $\bm{u}_1$ satisfies the estimate
$$||\bm{u}_1||_{H_0^1(D)^d}\leq C ||\operatorname{div}\bm{u}||_{L^2(D)},$$
and now the vector field $\bm{u}-\bm{u}_1$ is divergence free and satisfies
\begin{equation}\nonumber
\begin{split}
||\bm{e}(\bm{u})-\bm{e}(\bm{u}_1)||_{L^2(D_r^\e)^{d\times d}}&\leq C\left(||\bm{e}(\bm{u})||_{L^2(D_r^\e)^{d\times d}}+||\bm{u}_1||_{H_0^1(D)^d}\right)\\
&\leq C\left(||\bm{e}(\bm{u})||_{L^2(D_r^\e)^{d\times d}}+||\operatorname{div}\bm{u}||_{L^2(D)}\right).
\end{split}
\end{equation}
For the divergence free vector field $(\bm{u}-\bm{u}_1)_{|\e k+\e T}$ we apply a rescaled version of {\bf Lemma \ref{lemmaextension}} and obtain the divergence free vector field $\bm{u}_{2,k}\in H_0^1(\e k+\e Y)$, such that
$$\bm{e}(\bm{u}_{2,k})=\bm{e}(\bm{u}-\bm{u}_1)\ \mbox{ in } \e k+\e T$$
and
$$||\bm{u}_{2,k}||_{L^2(\e k+\e Y)^d}+\e ||\bm{\nabla}\bm{u}_{2,k}||_{L^2(\e k+\e Y)^{d\times d}}\leq C \e ||\bm{e}(\bm{u}-\bm{u}_1)||_{L^2(\e k+\e T)^{d\times d}} \ ,$$
where the constant $C$ depends only on $T$.\\
We form the sum of all $\bm{u}_{2,k}$ for all $k\in N^\e$ and obtain
$\bm{u}_2=\displaystyle\sum_{k\in N^\e}\bm{u}_{2,k}$, with $\bm{u}_2\in H_0^1(D)^d$ and
$$||\bm{u}_2||_{H_0^1(D)^d}\leq C ||\bm{e}(\bm{u}-\bm{u}_1)||_{L^2(D_r^\e)^{d\times d}}.$$
Take $ {\bm{u'}}=\bm{u}-\bm{u}_1-\bm{u}_2$ and see that
$$||\bm{u}- {\bm{u'}}||_{H_0^1(D)^d}\leq ||\bm{u}_1||_{H_0^1(D)^d}+||\bm{u}_2||_{H_0^1(D)^d}\leq C (||\operatorname{div}\bm{u}||_{L^2(D)}+||\bm{e}(\bm{u}-\bm{u}_1)||_{L^2(D_r^\e)^{d\times d}})$$
$$\leq C(||\bm{e}(\bm{u})||_{L^2(D_r^\e)^{d\times d}}+||\operatorname{div}\bm{u}||_{L^2(D)}),$$
and by construction, $ {\bm{u'}}$ is divergence free and $\bm{e}({\bm{u'}})=\bm{0}$ almost everywhere on $D_r^\e$, hence $ {\bm{u'}}\in V^\e$.
\end{proof}

\begin{lemma}
\label{lemmasuf}
Let $\bm{u}\in V$ and $\omega\in \Omega$. Assume that for any $\delta>0$ there exists a sequence $\bm{u}^\e_\delta\in H_0^1(D)^d$ weakly convergent to $\bm{u}$ as $\e\to 0$ such that
\begin{equation}
\label{c1}
\begin{split}
\limsup_{\delta\to 0}\limsup_{\e\to 0}\displaystyle\int_D(\operatorname{div}\bm{u}^\e_\delta)^2 d\bm{x}=\limsup_{\delta\to 0}\limsup_{\e\to 0}\int_{D_r^\e}\bm{e}(\bm{u}^\e_\delta):\bm{e}(\bm{u}^\e_\delta) d\bm{x}\\
=\limsup_{\delta\to 0}\limsup_{\e\to 0}|E^\e_\omega(\bm{u}^\e_\delta)-L|=0,
\end{split}
\end{equation}
for some $L\in\Re$. Then there exists a sequence $\bm{u}^\e\in V^\e$, weakly convergent to $\bm{u}$ as $\e\to 0$ such that $E^\e_\omega(\bm{u}^\e)\to I$ as $\e\to 0$.
\end{lemma}
\begin{proof}
For each vector field $\bm{u}^\e_\delta$ let $ {\bm{u'}}_\delta^\e$ be the vector field in $V^\e$ given by {\bf Lemma \ref{lemmaconstruction}} that satisfies
$$||\bm{u}^\e_\delta- {{\bm{u'}}}^\e_\delta||_{H_0^1(D)^d}\leq C\left(||\operatorname{div}\bm{u}^\e_\delta||_{L^2(D)}+||\bm{e}(\bm{u}^\e_\delta)||_{L^2(D_r^\e)^{d\times d}}\right).$$
Condition (\ref{c1}) implies that
$$\limsup_{\delta\to 0}\limsup_{\e\to 0} ||\bm{u}^\e_\delta- {{\bm{u'}}}^\e_\delta||_{H_0^1(D)^d}=0,$$
and using a diagonalization argument (\cite{attouch}, Corollary 1.18) we find a sequence $\delta(\e)\to 0$ such that
$$\lim_{\e\to 0} ||\bm{u}^\e_{\delta(\e)}- {\bm{u'}}^\e_{\delta(\e)}||_{H_0^1(D)^d}=0.$$
Then, the sequence $\bm{u}^\e= {\bm{u'}}^\e_{\delta(\e)}$ is the one we are looking for. 
\end{proof}
In the second part of the $\Gamma$-- convergence result we need to show that the sequence constructed is optimal, i.e. the limit given for any other sequence is greater, and we will use the next lemma to prove this.
\begin{lemma}
\label{lemmanec}
Let $\bm{A}\in\mc{S}$ with zero trace and let $\bm{w}^\e=\bm{A}x-\e\bm{\chi}_ {\bm{A}} \left(\dfrac{\cdot}{\e}\right)$ extended to the whole $D$ by periodicity. Then, for every sequence $\bm{u}^\e\in V^\e$: $\bm{u}^\e \rightharpoonup\bm{0}$ as $\e\to 0$, one has
$$\limsup_{\e\to 0} \int_{D}\phi\bm{e}(\bm{w}^\e):\bm{e}(\bm{u}^\e)d\bm{x}=0,$$
where $\phi\in W^{1,\infty}(D)$ compactly supported in $D$.
\end{lemma}
\begin{proof}
We have that
$$\bm{e}(\bm{w}^\e):\bm{e}(\phi\bm{u}^\e)=\phi\bm{e}(\bm{w}^\e):\bm{e}(\bm{u}^\e)+\displaystyle\sum_{i,j=1}^d e_{ij}(\bm{w}^\e)\left(\dfrac{\partial\phi}{\partial x_i}\bm{u}^\e_j+\dfrac{\partial\phi}{\partial x_j}\bm{u}^\e_i\right),$$
then
$$\limsup_{\e\to 0} \int_{U}\phi\bm{e}(\bm{w}^\e):\bm{e}(\bm{u}^\e)d\bm{x}=\limsup_{\e\to 0} \int_{U}\bm{e}(\bm{w}^\e):\bm{e}(\phi\bm{u}^\e)d\bm{x}.$$
The sequence $\phi\bm{u}^\e$ is weakly convergent to $\bm{0}$ and 
$$||\operatorname{div}\phi\bm{u}^\e||_{L^2(D)}\leq C||\bm{\nabla} \phi||_{L^\infty(D)^d}||\bm{u}^\e||_{L^2(D)^d},$$
$$\int_{D_r^\e}\bm{e}(\phi\bm{u}^\e):\bm{e}(\phi\bm{u}^\e) d\bm{x}\leq C ||\bm{\nabla} \phi||^2_{L^\infty(D)^d} ||\bm{u}^\e||^2_{L^2(D)^d}.$$
For every $\phi\bm{u}^\e$ we apply {\bf Lemma \ref{lemmaconstruction}} and construct $\bm{v}^\e \in V^\e$ such that
$$||\phi\bm{u}^\e-\bm{v}^\e||_{H_0^1(D)^d}\leq C\left(||\operatorname{div}\bm{u}||_{L^2(D)}+||\bm{e}(\bm{u})||_{L^2(D_r^\e)^{d\times d}}\right)\leq C||\bm{\nabla} \phi||_{L^\infty(D)^d}||\bm{u}^\e||_{L^2(D)^d},$$
with a constant $C$ independent of $\e$.
We notice that due to {\bf Theorem \ref{thcellglob}}
$$\int_{D}\bm{e}(\bm{w}^\e):\bm{e}(\bm{v}^\e)d\bm{x}=0,$$
and using the strong convergences to $\bm{0}$ of $\phi\bm{u}^\e-\bm{v}^\e$ we obtain that
$$\limsup_{\e\to 0} \int_{D}\phi\bm{e}(\bm{w}^\e):\bm{e}(\bm{u}^\e)d\bm{x}=\limsup_{\e\to 0} \int_{D}\bm{e}(\bm{w}^\e):\bm{e}(\phi\bm{u}^\e)d\bm{x}=0.$$
\end{proof}

We derive in the next lemma an important estimate satisfied by the functions $g^\e_k$ used in the definition (\ref{deff^e_k}) of the random superficial forces.

\begin{lemma}
\label{lemmaestg}
Let the functions $g^\e_k:\partial T^\e_k\times\Re^d\times\Omega\to \Re$ be given by (\ref{defg^e_k}). Then for every $\bm{u},\bm{v}\in H^1(D)^d$ the following estimate holds: 
\begin{equation}
\label{estg}
\int_{\partial D^\e_r} g^\e_k\left(\bm{x},\bm{u},\omega\right)ds-\int_{\partial D^\e_r} g^\e_k\left(\bm{x},\bm{v},\omega\right)ds\leq C^\e(\omega)\left(||\bm{u}-\bm{v}||_{L^{2}(D)^d}^\gamma+\e^{\gamma}||D\bm{u}-D\bm{v}||_{L^{2}(D)^{d\times d}}^\gamma\right),
\end{equation}
where there exists a constant $C$ such that $C^\e(\omega)\to C$ a.s. $\omega\in\Omega$ as $\e\to 0$ and $\gamma$ is the H\"{o}lder constant given by (\ref{h5}).
\end{lemma}
\begin{proof}
We use first the property (\ref{h5}) and notice that one can assume the function $a$ is nonnegative. We obtain that a.s. $\omega\in\Omega$ we have
\begin{equation}\nonumber
\begin{split}
\int_{\partial D^\e_r} g^\e_k\left(\bm{x},\bm{u},\omega\right)ds-\int_{\partial D^\e_r} g^\e_k\left(\bm{x},\bm{v},\omega\right)ds&=\sum_{k\in N^\e}\int_{\partial T^\e_k} \e\left( g\left(\frac{s}{\e},\bm{u},\tau_k\omega\right)-g\left(\frac{s}{\e},\bm{v},\tau_k\omega\right)\right)ds\\
&\leq \sum_{k\in N^\e} \int_{\partial T^\e_k} \e a\left(\frac{\bm{x}}{\e},\tau_k\omega\right)|\bm{u}-\bm{v}|^\gamma ds.
\end{split}
\end{equation}
Then from H\"older's inequality and the change of variables $\dfrac{\bm{x}}{\e}=s$ we obtain
\begin{equation}\nonumber
\begin{split}
\int_{\partial T^\e_k} a\left(\frac{\bm{x}}{\e},\tau_k\omega\right)|\bm{u}-\bm{v}|^\gamma ds&\leq \left(\int_{\partial T^\e_k} a\left(\frac{\bm{x}}{\e},\tau_k\omega\right)^{\frac{2}{2-\gamma}} ds\right)^{\frac{2-\gamma}{2}}\left(\int_{\partial T^\e_k} |\bm{u}-\bm{v}|^{2} ds\right)^{\frac{\gamma}{2}}\\
&=\e^{\frac{(n-1)(2-\gamma)}{2}}\left(\int_{\partial T} a\left(s,\tau_k\omega\right)^{\frac{2}{2-\gamma}} ds\right)^{\frac{2-\gamma}{2}}\left(\int_{\partial T^\e_k} |\bm{u}-\bm{v}|^{2} ds\right)^{\frac{\gamma}{2}},
\end{split}
\end{equation}
which implies that
\begin{equation}\nonumber
\begin{split}
&\int_{\partial D^\e_r} g^\e_k\left(\bm{x},\bm{u},\omega\right) ds-\int_{\partial D^\e_r} g^\e_k\left(\bm{x},\bm{v},\omega\right) ds\leq \sum_{k\in N^\e} \int_{\partial T^\e_k} \e a\left(\frac{\bm{x}}{\e},\tau_k\omega\right)|\bm{u}-\bm{v}|^\gamma ds\\
&\leq \e^{1+\frac{(n-1)(2-\gamma)}{2}}\sum_{k\in N^\e}\left(\int_{\partial T} a\left(s,\tau_k\omega\right)^{\frac{2}{2-\gamma}}ds \right)^{\frac{2-\gamma}{2}} \left(\int_{\partial T^\e_k} |\bm{u}-\bm{v}|^{2} ds\right)^{\frac{\gamma}{2}}\\
&\leq C \e^{1+\frac{(n-1)(2-\gamma)}{2}}\left(\sum_{k\in N^\e}\int_{\partial T} a\left(s,\tau_k\omega\right)^{\frac{2}{2-\gamma}} ds\right)^{\frac{2-\gamma}{2}}\left(\sum_{k\in N^\e}\int_{\partial T^\e_k} |\bm{u}-\bm{v}|^{2} ds\right)^{\frac{\gamma}{2}}.\\
\end{split}
\end{equation}
We obtain in the end that
\begin{equation}\nonumber
\begin{split}
&\int_{\partial D^\e_r} \e g\left(\frac{s}{\e},\bm{u},\omega\right)ds-\int_{\partial D^\e_r} \e g\left(\frac{s}{\e},\bm{v},\omega\right)ds\leq\\
&C\e^{1+\frac{(n-1)(2-\gamma)}{2}} |N^\e|^{\frac{2-\gamma}{2}}\left(\sum_{k\in N^\e}\frac{1}{|N^\e|}\int_{\partial T} a\left(s,\tau_k\omega\right)^{\frac{2}{2-\gamma}} ds\right)^{\frac{2-\gamma}{2}}\left(\sum_{k\in N^\e}\int_{\partial T^\e_k} |\bm{u}-\bm{v}|^{2} ds\right)^{\frac{\gamma}{2}}\\
&\leq C\e^\frac{\gamma}{2} \left(\sum_{k\in N^\e}\frac{1}{|N^\e|}\int_{\partial T} a\left(s,\tau_k\omega\right)^{\frac{2}{2-\gamma}}ds\right)^{\frac{2-\gamma}{2}}\left(\sum_{k\in N^\e}\int_{\partial T^\e_k} |\bm{u}-\bm{v}|^{2}ds\right)^{\frac{\gamma}{2}}\\
&\leq C\e^{\frac{\gamma}{2}} C^\e(\omega)\left(\sum_{k\in N^\e}\int_{\partial T^\e_k} |\bm{u}-\bm{v}|^{2} ds\right)^{\frac{\gamma}{2}},
\end{split}
\end{equation}
where, by the Ergodic Theorem $C^\e(\omega)\to ||a||_{L^{\frac{2}{2-\gamma}}(\partial T\times \Omega)}$ as $\e\to 0$ a.s. $\omega\in\Omega$, and from the property (H5), $||a||_{L^{\frac{2}{2-\gamma}}(\partial T\times \Omega)}$ is finite.

Applying a rescaled Poincar\'e inequality
$$\int_{\partial T^\e_k} |\bm{u}-\bm{v}|^{2} ds\leq \e^{-1}\left(\int_{T^\e_k}|\bm{u}-\bm{v}|^{2}d\bm{x}+\e^{2}\int_{T^\e_k}|D\bm{u}-D\bm{v}|^{2}d\bm{x}\right),$$
one has
$$\int_{\partial D^\e_r}\e g\left(\frac{s}{\e},\bm{u},\omega\right)ds-\int_{\partial D^\e_r} \e g\left(\frac{s}{\e},\bm{v},\omega\right)ds\leq C^\e(\omega) \left(\int_{D_r^\e}|\bm{u}-\bm{v}|^{2}d\bm{x}+\e^{2}\int_{D_r^\e}|D\bm{u}-D\bm{v}|^{2}d\bm{x}\right)^{\frac{\gamma}{2}},$$
which yield the desired estimate.

In order to prove our main convergence result we need an improved version of the pointwise Ergodic Theorem that we state and prove in the following theorem. This theorem can also be viewed as an extension of the pointwise Ergodic Theorem proved in \cite{marver2} with weights of the form $\displaystyle\dashint_{\e k+\e [0,1]^n}\bm{u}(\bm{x}) d\bm{x}$ added in computing the averages, where $\bm{u}\in L^1(U)$, and $U\subset \Re^d$ is open, bounded with Lipschitz boundary.
\end{proof}
\begin{theorem}
\label{therg}
Assume $1< p\leq +\infty$ and $h:\Re^m\times \Omega\to \Re$ a function such that $\omega\mapsto h(\bm{0},\omega)$ belongs to $L^1(\Omega)$ and for every $\bm{z}_1,\bm{z}_2\in \Re^m$ and a.s. $\omega\in\Omega$

$$h(\bm{z}_1,\omega)-h(\bm{z}_2,\omega)\leq |\bm{z}_1-\bm{z}_s|^\frac{p-1}{p} a(\omega),$$
for a function $a\in L^p(\Omega)$.

Then, on a set of full probability the sequence $\e^n\displaystyle\sum_{k\in \Int^n\cap \frac{U}{\e}}h\left(\dashint_{\e k+\e [0,1]^n}\bm{u}(\bm{x}) d\bm{x},\tau_k\omega\right)$ converges to $\displaystyle\int_U h(\bm{u}(\bm{x}),\omega) dP d\bm{x}$ for any $U\subset\Re^n$ open, bounded with Lipschitz boundary and any $\bm{u}\in L^1(U)^m$.
\end{theorem}
\begin{proof}
Obviously, for every $\bm{z}\in\Re^m$, the function $\omega\mapsto h(\bm{z},\omega)$ belongs to $L^1(\Omega)$. Let $\Omega'$ be a set of full probability such that for every $\omega\in\Omega '$ and for any $\bm{z}\in \Rat^m$, the sequence $\displaystyle\sum_{k\in \Int^n\cap \frac{U}{\e}} \e^n h(\bm{z},\tau_k \omega)$ converges to $\displaystyle\int_U\int_\Omega h(\bm{z},\omega) dP d\bm{x}$, for every $U\in\Re^n$, bounded with Lipschitz boundary. Moreover,
$$\sum_{k\in \Int^n\cap \frac{U}{\e}} \e^n h(\bm{z}_1,\tau_k \omega)-\sum_{k\in \Int^n\cap \frac{U}{\e}} \e^n h(\bm{z}_2,\tau_k \omega)\leq |\bm{z}_1-\bm{z}_2|^\frac{p-1}{p}\sum_{k\in \Int^n\cap \frac{U}{\e}} \e^n a(\tau_k \omega),$$
therefore passing to the limit $\e^n\displaystyle\sum_{k\in \Int^n\cap \frac{U}{\e}}h\left(\bm{z},\tau_k\omega\right) d\bm{x}$ converges to $\displaystyle\int_U\int_\Omega h(\bm{z},\omega) dP d\bm{x}$ for any $U\subset\Re^n$ open, bounded with Lipschitz boundary, any $\bm{z}\in \Re^m$, and a.s. $\omega\in\Omega$.

Assuming now that $\bm{u}\in C(\overline{U})^m$, where $U$ is as in the assumptions of the theorem, and let $\bm{x}_0\in \Re^n$ be such that
$$\int_U\int_\Omega h(\bm{u},\omega) dP d\bm{x}=\int_U\int_\Omega h(\bm{u}(\bm{x}_0),\omega) dP d\bm{x}.$$
Then, 
$$\limsup_{\e\to 0}\left|\sum_{k\in \Int^n\cap \frac{U}{\e}}\e^n h(\bm{u}(\e k),\tau_k\omega)-\int_U\int_\Omega h(\bm{u},\omega) dP d\bm{x}\right|\leq ||a||_{L^1(\Omega)}(\sup_U\bm{u}-\inf_U\bm{u})^\frac{p-1}{p}\mc{L}^n(U),$$
a.s.. The continuity of $\bm{u}$ implies that
$$\lim_{\e\to 0}\sum_{k\in \Int^n\cap \frac{U}{\e}}\e^n h(\bm{u}(\e k),\tau_k\omega)=\int_U\int_\Omega h(\bm{u},\omega) dP d\bm{x},$$
a.s. $\omega\in\Omega$. In order to prove the theorem, we use the separability of $L^1(U)^m$ and the density of $C(\overline{U})^m$ in $L^1(U)^m$. We first notice that we can show in the same way that 
$$\lim_{\e\to 0}\sum_{k\in \Int^n\cap \frac{U}{\e}}\e^n h\left(\dashint_{\e k+\e [0,1]^n}\bm{u} d\bm{x},\tau_k\omega\right)=\int_U\int_\Omega h(\bm{u},\omega) dP d\bm{x},$$
for $\bm{u}\in C(\overline{U})^m$. If $\bm{u}\in L^1(U)^m$ and $\bm{v}\in C(\overline{U})^m$,
$$\sum_{k\in \Int^n\cap \frac{U}{\e}}\e^n h\left(\dashint_{\e k+\e [0,1]^n}\bm{u} d\bm{x},\tau_k\omega\right)-\sum_{k\in \Int^n\cap \frac{U}{\e}}\e^n h\left(\dashint_{\e k+\e [0,1]^n}\bm{v} d\bm{x},\tau_k\omega\right)\leq$$
$$\sum_{k\in \Int^n\cap \frac{U}{\e}}\e^n \left(\dashint_{\e k+\e [0,1]^n}(\bm{u}-\bm{v}) d\bm{x}\right)^{\frac{p-1}{p}} a(\tau_k\omega)\leq$$
$$\e^{np} \left(\sum_{k\in \Int^n\cap \frac{U}{\e}}\int_{\e k+\e [0,1]^n}(\bm{u}-\bm{v}) d\bm{x}\right)^{\frac{p-1}{p}} \left(\sum_{k\in \Int^n\cap \frac{U}{\e}} a(\tau_k\omega)^p\right)^\frac{1}{p}.$$
This implies that a.s. $\omega\in\Omega$
$$\limsup_{\e\to 0}\left|\sum_{k\in \Int^n\cap \frac{U}{\e}}\e^n h\left(\dashint_{\e k+\e [0,1]^n}\bm{u} d\bm{x},\tau_k\omega\right)-\int_U\int_\Omega h(\bm{v},\omega) dP d\bm{x}\right|\leq C||\bm{u}-\bm{v}||_{L^1(U)^m}^{\frac{p-1}{p}}.$$
Now we let $\bm{v}$ converge to $\bm{u}$ in $L^1(U)^m$ and obtain the result.
\end{proof}
We use the Ergodic result given by {\bf Theorem \ref{therg}} and the estimate given by {\bf Lemma \ref{lemmaestg}} to derive the following two limits, that essentially show how the random, highly oscillating forces acting on the boundaries of the rigid particles, become through homogenization, deterministic forces in the volume.
\begin{theorem}
\label{therge}
Let $\bm{u}$ and $\bm{v}$ be vector fields in $H^1(D)^d$ and let $\bm{u}^\e$ and $\bm{v}^\e$ be two sequences such that $\bm{u}^\e\rightharpoonup\bm{u}$ and $\bm{u}^\e \rightharpoonup\bm{v}$ in $H^1(D)^d$ as $\e\to 0$. Then, on a set of full probability we have:

i) $\displaystyle\lim_{\e\to 0}\sum_{k\in N^\e}\int_{\partial T^\e_k}\e g\left(\frac{\cdot}{\e},\bm{u}^\e,\tau_k\omega\right) ds = \int_D\int_{\partial T}\int_\Omega g\left(s,\bm{u},\omega\right) dP ds d\bm{x}.$

ii) $\displaystyle\lim_{\e\to 0}\sum_{k\in N^\e}\int_{\partial T^\e_k}\e\bm{\nabla}_{\bm{u}}g\left(\frac{\cdot}{\e},\bm{u}^\e,\tau_k\omega\right)\cdot\bm{v}^\e ds = \int_D\int_{\partial T}\int_\Omega\bm{\nabla}_{\bm{u}}g\left(s,\bm{u},\omega\right)\cdot\bm{v} dP ds d\bm{x}.$ 
\end{theorem}
\begin{proof}
i) According to {\bf Lemma \ref{lemmaestg}} 
\begin{equation}
\label{estA}
\begin{split}
&\limsup_{\e\to 0}\left|\sum_{k\in N^\e}\int_{\partial T^\e_k}\e g\left(\frac{\cdot}{\e},\bm{u}^\e,\tau_k\omega\right) ds-\sum_{k\in N^\e}\int_{\partial T^\e_k}\e g\left(\frac{\cdot}{\e},\bm{u},\tau_k\omega\right) ds\right|\\
&\leq \limsup_{\e\to 0}C^\e(\omega)\left(||\bm{u}^\e-\bm{u}||_{L^{2}(D)^d}^\gamma+\e^{\gamma}||\bm{\nabla}\bm{u}^\e-\bm{\nabla}\bm{u}||_{L^{2}(D)^{d\times d}}^\gamma\right)\\
&\leq C \limsup_{\e\to 0}||\bm{u}^\e_\delta-\bm{u}||_{L^{2}(D)^d}^\gamma=0,\mbox{ a.s. } \omega\in\Omega.
\end{split}
\end{equation}
We will show now that 
\begin{equation}
\label{estB}
B=\displaystyle\limsup_{\e\to 0}\left|\sum_{k\in N^\e}\int_{\partial T^\e_k}\e g\left(\frac{\cdot}{\e},\bm{u},\tau_k\omega\right)ds-\int_D\int_{\partial T}\int_{\Omega}g(s,\bm{u},\omega) dP ds d\bm{x}\ \right|=0.
\end{equation}
First, the integral $\displaystyle\sum_{k\in N^\e}\int_{\partial T^\e_k}\e g\left(\frac{\cdot}{\e},\bm{u},\tau_k\omega\right)ds$, after a change of variables, becomes $$\displaystyle\sum_{k\in N^\e}\e^d\int_{\partial T} g\left(k+s,\bm{u}(\e k+\e s),\tau_k\omega\right)ds=\displaystyle\sum_{k\in N^\e}\e^d\int_{\partial T} g\left(s,\bm{u}(\e k+\e s),\tau_k\omega\right)ds.$$ Denote by $\bm{u}^\e_k= \displaystyle\dashint_{\e k+\e Y}\bm{u}d\bm{x}$, then
$$\int_{\partial T} g\left(s,\bm{u}(\e k+\e s),\tau_k\omega\right)ds-\int_{\partial T} g\left(\cdot,\bm{u}^\e_k,\tau_k\omega\right)ds\leq C(\e)\int_{\partial T} a(s,\tau_k\omega)ds,$$
where $C(\e) \to 0$ as $\e\to 0$. Using the Ergodic Theorem, we obtain that a.s. $\omega\in\Omega$
$$\displaystyle\lim_{\e\to 0}\left[\sum_{k\in N^\e}\e^d\int_{\partial T} g\left(s,\bm{u}(\e k+\e s),\tau_k\omega\right)ds-\sum_{k\in N^\e}\e^d\int_{\partial T} g\left(\cdot,\bm{u}^\e_k,\tau_k\omega\right)ds\right] =0.$$
So
\begin{equation}\nonumber
\begin{split}
B=&\limsup_{\e\to 0}\left|\sum_{k\in N^\e}\int_{\partial T}\e^d g\left(\cdot,\bm{u}^\e_k,\tau_k\omega\right)ds-\int_D\int_{\partial T}\int_{\Omega}g(s,\bm{u},\omega) dP d\bm{x} ds\ \right| \\
=&\limsup_{\e\to 0}\left|\sum_{k\in N^\e}\e^d \tilde{g}\left(\dashint_{\e k+\e Y}\bm{u} d\bm{x},\tau_k\omega\right)-\int_D\int_{\Omega} \tilde{g}(\bm{u},\omega) dP d\bm{x}\ \right|,
\end{split}
\end{equation}
where $\tilde{g}$ is defined on $\Re^d\times\Omega$, by $\tilde{g}(\bm{z},\omega)=\displaystyle\int_{\partial T} g(s,\bm{z},\omega) ds$. From {\bf Theorem \ref{therg}} we obtain that $B=0$.

ii) Using the definition of the subdifferential we have that for every $t\in\Re$
$$t\bm{\nabla}_{\bm{u}}g\left(\frac{s}{\e},\bm{u}^\e,\tau_k\omega\right)\cdot\bm{v}^\e\leq
g\left(\frac{s}{\e},\bm{u}^\e+t\bm{v}^\e,\tau_k\omega\right)-g\left(\frac{s}{\e},\bm{u}^\e,\tau_k\omega\right),$$
therefore
$$t\sum_{k\in N^\e}\int_{\partial T^\e_k}\e\bm{\nabla}_{\bm{u}}g\left(\frac{\cdot}{\e},\bm{u}^\e,\tau_k\omega\right)\cdot\bm{v}^\e ds\leq \sum_{k\in N^\e}\int_{\partial T^\e_k}\left(\e g\left(\frac{\cdot}{\e},\bm{u}^\e+t\bm{v}^\e,\tau_k\omega\right) -\e g\left(\frac{\cdot}{\e},\bm{u}^\e,\tau_k\omega\right)\right) ds.$$
Using the first part, we obtain that
\begin{equation}\nonumber
\begin{split}
\limsup_{\e\to 0}t\sum_{k\in N^\e}\int_{\partial T^\e_k}\e\bm{\nabla}_{\bm{u}}g\left(\frac{\cdot}{\e},\bm{u}^\e,\tau_k\omega\right)\cdot\bm{v}^\e ds&\leq \int_D\int_{\partial T}\int_\Omega \left(g\left(s,\bm{u}+t\bm{v},\omega\right) -g\left(s,\bm{u},\omega\right)\right)dP ds d\bm{x}\\
&\leq \int_D\int_{\partial T}\int_\Omega\bm{\nabla}_{\bm{u}}g\left(s,\bm{u}+t\bm{v},\omega\right)\cdot t\bm{v} dP ds d\bm{x}.
\end{split}
\end{equation}
We take first $t>0$, and then $t<0$ and let it converge to $0$, and then apply the Dominated Convergence Theorem and obtain that
\begin{equation}\nonumber
\begin{split}
\lim_{\e\to 0}\sum_{k\in N^\e}\int_{\partial T^\e_k}\e\bm{\nabla}_{\bm{u}}g\left(\frac{\cdot}{\e},\bm{u}^\e,\tau_k\omega\right)\cdot\bm{v}^\e ds=&\lim_{t\to 0}\int_D\int_{\partial T}\int_\Omega\bm{\nabla}_{\bm{u}}g\left(s,\bm{u}+t\bm{v},\omega\right)\cdot\bm{v} dP ds d\bm{x}\\
&=\int_D\int_{\partial T}\int_\Omega\bm{\nabla}_{\bm{u}}g\left(s,\bm{u},\omega\right)\cdot\bm{v} dP ds d\bm{x}.
\end{split}
\end{equation}
\end{proof}
\subsection{Convergence Results}
\label{ssec4.3}
Assume there exists $\bm{f}\in L^2(D)^d$ such that $\bm{f}^\e\rightharpoonup\bm{f}$ and introduce now the functional $E^*:H^1(D)^d\to \Re\cup\{+\infty\}$ defined by
\begin{equation}
\label{defE}
E^*(\bm{v})=\int_{D} \left(\mu\bm{e}(\bm{v}):\bm{e}(\bm{v})+\mc{C}[\bm{e}(\bm{v}),\bm{e}(\bm{v})]\right)d\bm{x}-\int_D\bm{f}\cdot\bm{v}d\bm{x}+\int_D\int_{\partial T}\int_{\Omega}g(s,\bm{v},\omega) dPdsd\bm{x}+I_V(\bm{v}).
\end{equation}
Properties (H4) and (H5) of the function $g$ imply that there exists a constant $C$ such that
\begin{equation}
\label{coercE}
\begin{split}
\left|\int_{\partial T}\int_{\Omega}g(s,\bm{z},\omega) dPds\right|&\leq \left|\int_{\partial T}\int_{\Omega}g(s,\bm{0},\omega) dPds\right|+C|\bm{z}|^\gamma \\
&\leq C+C|\bm{z}|^\gamma,
\end{split}
\end{equation}
for every $\bm{z}\in\Re^d$. Thus $E^*$ is well defined on $H^1(D)^d$ as a consequence of the property (H2), is strictly convex and takes finite values on $V$. Moreover (\ref{coercE}) implies that there exist two constants $A>0$, and $C$ such that for all $\bm{v}\in H^1(D)^d$ we have 
$$E^*(\bm{v})\geq -C - C||\bm{v}||_{L^2(D)^d}^\gamma-C||\bm{v}||_{L^2(D)^d} +A||\bm{v}||_{L^2(D)^d}^2,$$
therefore there exists an unique minimizer $\bm{u}^*\in V$ for the functional $E^*$. It can be shown similarly as in {\bf Theorem \ref{thminfor}} that the minimizer $\bm{u}^*$ will also satisfy the following equivalent weak formulation:
\begin{equation}
\label{eqvarforh2}
\int_{D} 2\mu\bm{e}(\bm{u}^*):\bm{e}(\bm{\bm{\phi}}) d\bm{x}+\int_D 2\mc{C}[\bm{e}(\bm{u}^*),\bm{e}(\bm{\phi})]d\bm{x}=\int_D\bm{f}\cdot\bm{\phi}d\bm{x}+\int_D\bm{f}^*(\bm{u}^*)\cdot\bm{\phi} d\bm{x},
\end{equation}
for every $\bm{\phi}\in V$, where we denoted by $\bm{f}^*$ the vector field on $\Re^d$ defined by
\begin{equation}
\label{defF}
\bm{f}^*(\bm{z})=-\int_{\partial T}\int_{\Omega}\bm{\nabla}_{\bm{z}}g(s,\bm{z},\omega)dPds.
\end{equation}
The linear functional defined on $H^1(D)^d$
$$\bm{\phi\mapsto} \int_{D} 2\mu\bm{e}(\bm{u}^*):\bm{e}(\bm{\bm{\phi}}) d\bm{x}+\int_D 2\mc{C}[\bm{e}(\bm{u}^*),\bm{e}(\bm{\phi})]d\bm{x}-\int_D\bm{f}\cdot\bm{\phi}d\bm{x}-\int_D\bm{f}^*(\bm{u}^*)\cdot\bm{\phi} d\bm{x}$$
being $0$ on $V$, implies the existence of $p^*\in L^2(D)$ such that
\begin{equation}
\label{eqvarforh1}
\int_{D} 2\mu\bm{e}(\bm{u}^*):\bm{e}(\bm{\bm{\phi}}) d\bm{x}+\int_D 2\mc{C}[\bm{e}(\bm{u}^*),\bm{e}(\bm{\phi})]d\bm{x}-\int_D p^*\operatorname{div}\bm{\phi} d\bm{x}=
\int_D\bm{f}\cdot\bm{\phi}d\bm{x}+\int_D\bm{f}^*(\bm{u}^*)\cdot\bm{\phi} d\bm{x},
\end{equation}
for every $\bm{\phi}\in H_0^1(D)^d$.

The pair $\{\bm{u}^*,p^*\}$, $\bm{u}^*\in H^1(D)^d$, $p^*\in L^2(D)$, satisfying equation (\ref{eqvarforh1}), is a solution for a Stokes system that we describe now. Given the bilinear form $\mc{C}$, we denote by $\bm{\mu}^*\bm{A}$ the unique element from $\mc{S}$ such that
\begin{equation}
\label{defmu^*}
\mu\bm{A}:\bm{B}+\mc{C}[\bm{A},\bm{B}]=\bm{\mu}^*\bm{A}:\bm{B},
\end{equation}
for all $\bm{B}\in\mc{S}$. We also denote for any $\bm{u}\in H^1(D)^d$ and any $p\in L^2(D)$ by $\bm{\sigma}^*(\bm{u},p)$ the stress tensor
\begin{equation}
\label{defsigma^*}
\bm{\sigma}^*(\bm{u},p)=2\bm{\mu}^*\bm{e}(\bm{u})-p\bm{I}.
\end{equation}
Equation (\ref{eqvarforh2}) can be written equivalently in the following way: $\{\bm{u}^*,p^*\}$ is the solution for the Stokes system
\begin{equation}
\label{systemh}
\left\{
\begin{array}{rll}
-\operatorname{div}\bm{\sigma}^*(\bm{u},p)&=\bm{f}+\bm{f}^*(\bm{u}) &\ \mbox{in} \ D, \\
\operatorname{div}\bm{u} &=0 &\ \mbox{in}\ D, \\
\bm{u} &=\bm{0} &\ \mbox{on}\ \partial D.
\end{array}
\right.
\end{equation}

We will show next that $E^*$ is the limit functional for the sequence of functionals $(E^\e_\omega)_{\e >0}$ and $\bm{u}^*$, the unique minimizer for $E^*$ is the limit solution, i.e. $E^*$ is the $\Gamma$-- limit for the sequence $(E^\e_\omega)_{\e >0}$ and $\bm{u}^\e_\omega\rightharpoonup\bm{u}^*$ in $V$, a.s. $\omega\in\Omega$. Then, we have the following result:
\begin{theorem}
\label{thGammaconv}
The sequence of functionals $(E^\e_\omega)_{\e >0} $ $\Gamma$-- converges to the functional $E^*$ in the weak topology of $H^1(D)^d$ a.s. $\omega\in\Omega$, where $E^*:H^1(D)^d\to \Re\cup\{+\infty\}$ is defined by (\ref{defE}).
\end{theorem}
\begin{proof}
i) For the first part we need to show that
$$\Gamma-\limsup_{\e\to 0}E^\e_\omega\leq E^*,$$
a.s. $\omega\in\Omega$.

We take first $\bm{u}\in V\cap W^{2,\infty}(D)^d$ and let $\delta>0$. Then, there exists a partition of $D$, $(D_i)_{i=1}^{N_\delta}$, such that for all $1\leq i \leq N_\delta$, $D_i$ is Lipschitz and
$$\mathop{\essosc}_{D_i}\bm{e}(\bm{u})\leq \delta,$$
where $\displaystyle\mathop{\essosc}_{D_i}\bm{e}(\bm{u})$ denotes the essential oscillation of $\bm{e}(\bm{u})$ over the set $D_i$, i.e., the smallest number $t\geq 0$ such that $||\bm{e}(\bm{u}(\bm{x}_1))-\bm{e}(\bm{u}(\bm{x}_2))||\leq t$, for almost all $\bm{x}_1$ and $\bm{x}_2$ in $D_i$.

We define $\bm{A}_i=\displaystyle\dashint_{D_i}\bm{e}(\bm{u})d\bm{x}$ and let $\bm{u}^\e_i(\bm{x})=\e\bm{\chi}_{\bm{A}_i}\left(\dfrac{\bm{x}}{\e}\right)$ extended by periodicity to the whole $D_i$. For every $1\leq i\leq N_\delta$ we denote by $D_{i,\delta}$ the set of points from $D_i$ with the distance to the boundary of $D_i$ at least $\delta$, and let $\phi_{i,\delta}$ be a smooth cutoff function supported in $D_i$, such that $0\leq\phi_{i,\delta}\leq 1$ in $D_i$, equal to 1 in $D_{i,\delta}$ and such that $||\bm{\nabla} \phi_{i,\delta}||_{L^\infty(D_i)^d}\leq \dfrac{C}{\delta}$. Let us define $\bm{u}^\e_\delta$ by
$$\bm{u}^\e_\delta=\left(\bm{u}-\sum_{i}\phi_{i,\delta}\bm{u}^\e_i\right) \rightharpoonup\bm{u}\mbox{ as } \e\to 0.$$
We now show that the sequence $\bm{u}^\e_\delta$ satisfies properties (\ref{c1}) a.s. $\omega\in\Omega$. It is easy to check that $||\operatorname{div}\bm{u}^\e_\delta||_{L^2(D)}\to 0$ as $\e\to 0$. Also
$$\int_{D_r^\e}\bm{e}(\bm{u}^\e_\delta):\bm{e}(\bm{u}^\e_\delta) d\bm{x}=\sum_{i=1}^{N_\delta}\int_{D_r^\e\cap D_i}\bm{e}(\bm{u}^\e_\delta):\bm{e}(\bm{u}^\e_\delta) d\bm{x}$$
$$\leq\sum_{i=1}^{N_\delta}2 \int_{D_r^\e\cap D_i}\bm{e}(\bm{u}-\bm{u}^\e_i):\bm{e}(\bm{u}-\bm{u}^\e_i) d\bm{x}+ \sum_{i=1}^{N_\delta}2 \int_{D_r^\e\cap D_i}\bm{e}(\bm{u}^\e_i(1-\phi_{i,\delta})):\bm{e}(\bm{u}^\e_i(1-\phi_{i,\delta})) d\bm{x}$$
$$\leq C\int_D \delta^2 d\bm{x}+\sum_{i=1}^{N_\delta} C\int_{D_r^\e\cap D_i}\bm{e}(\bm{u}^\e_i):\bm{e}(\bm{u}^\e_i)(1-\phi_{i,\delta})^2 d\bm{x}+\sum_{i=1}^{N_\delta} C\int_{D_i}|\bm{u}^\e_i|^2 |\bm{\nabla} \phi_{i,\delta}|^2d\bm{x}$$
$$\leq C\delta^2+\sum_{i=1}^{N_\delta} C\int_{D_i\setminus D_{i,\delta}}\bm{e}(\bm{u}^\e_i):\bm{e}(\bm{u}^\e_i) d\bm{x}+\sum_{i=1}^{N_\delta}\dfrac{C}{\delta^2}\int_{D_i}|\bm{u}^\e_i|^2 d\bm{x},$$
which implies that
$$\limsup_{\delta\to 0}\limsup_{\e\to 0} \int_{D^\e_r}\bm{e}(\bm{u}^\e_\delta):\bm{e}(\bm{u}^\e_\delta) d\bm{x}=0.$$
Finally,
$$E^\e_\omega(\bm{u}^\e_\delta)-E^*(\bm{u})=\int_D \mu\bm{e}(\bm{u}^\e_\delta):\bm{e}(\bm{u}^\e_\delta)d\bm{x}-\int_D\bm{f}^\e\cdot\bm{u}^\e_\delta d\bm{x}+\sum_{k\in N^\e}\int_{\partial T^\e_k}\e g\left(\frac{s}{\e},\bm{u}^\e_\delta,\tau_k\omega\right)ds$$
$$-\int_{D} \left(\mu\bm{e}(\bm{u}):\bm{e}(\bm{u})+\mc{C}[\bm{e}(\bm{u}),\bm{e}(\bm{u})]\right)d\bm{x}+\int_D\bm{f}\cdot\bm{u}d\bm{x}-\int_D\int_{\partial T}\int_{\Omega}g(s,\bm{u},\omega) dP ds d\bm{x},$$
hencefore
$$\limsup_{\e\to 0} |E^\e_\omega(\bm{u}^\e_\delta)-E^*(\bm{u})|\leq I_{1,\delta}+I_{2,\delta}+I_{3,\delta},$$
where
$$I_{1,\delta}=\limsup_{\e\to 0}\left|\int_D \mu\bm{e}(\bm{u}^\e_\delta):\bm{e}(\bm{u}^\e_\delta)d\bm{x}-\int_{D} \left(\mu\bm{e}(\bm{u}):\bm{e}(\bm{u})+\mc{C}[\bm{e}(\bm{u}),\bm{e}(\bm{u})]\right)d\bm{x}\ \right|,$$
$$I_{2,\delta}=\limsup_{\e\to 0}\left|\int_D\bm{f}^\e\cdot\bm{u}^\e_\delta d\bm{x}-\int_D\bm{f}\cdot\bm{u}d\bm{x}\ \right|,$$
and
$$I_{3,\delta}=\limsup_{\e\to 0}\left|\sum_{k\in N^\e}\int_{\partial T^\e_k}\e g\left(\frac{s}{\e},\bm{u}^\e_\delta,\tau_k\omega\right)ds-\int_D\int_{\partial T}\int_{\Omega}g(s,\bm{u},\omega) dP d\bm{x} ds\ \right|.$$

It is easy to see that $I_{2,\delta}=0$, and according to the first part of {\bf Theorem \ref{therge}} $I_{3,\delta}=0$ as well.

To estimate $I_{1,\delta}$ we make the following elementary calculations:
$$\int_D \mu\bm{e}(\bm{u}^\e_\delta):\bm{e}(\bm{u}^\e_\delta)d\bm{x}-\int_{D}\left(\mu\bm{e}(\bm{u}):\bm{e}(\bm{u})+\mc{C}[\bm{e}(\bm{u}),\bm{e}(\bm{u})]\right)d\bm{x}=$$
$$\sum_{i=1}^{N_\delta} \left(\int_{D_i} \mu\bm{e}(\bm{u}-\phi_{i,\delta}\bm{u}^\e_i):\bm{e}(\bm{u}-\phi_{i,\delta}\bm{u}^\e_i)d\bm{x}-\int_{D_i} \left(\mu\bm{e}(\bm{u}):\bm{e}(\bm{u})+\mc{C}[\bm{e}(\bm{u}),\bm{e}(\bm{u})]\right)d\bm{x} \right)=$$
$$\sum_{i=1}^{N_\delta} \left(\int_{D_i} -2\mu\bm{e}(\bm{u}):\bm{e}(\phi_{i,\delta}\bm{u}^\e_i)d\bm{x}+ \int_{D_i} \mu\bm{e}(\phi_{i,\delta}\bm{u}^\e_i):\bm{e}(\phi_{i,\delta}\bm{u}^\e_i)d\bm{x}-\int_{D_i}\mc{C}[\bm{e}(\bm{u}),\bm{e}(\bm{u})]d\bm{x} \right).$$
Since $\phi_{i,\delta}\bm{u}^\e_i \rightharpoonup\bm{0}$ as $\e\to 0$,
$$I_{1,\delta}\leq\sum_{i=1}^{N_\delta}\limsup_{\e\to 0}\left|\int_{D_i} \mu\bm{e}(\phi_{i,\delta}\bm{u}^\e_i):\bm{e}(\phi_{i,\delta}\bm{u}^\e_i)d\bm{x}-\int_{D_i}\mc{C}[\bm{e}(\bm{u}),\bm{e}(\bm{u})]d\bm{x} \right|\leq$$
$$\sum_{i=1}^{N_\delta}\limsup_{\e\to 0}\left|\int_{D_i} \mu \phi_{i,\delta}\bm{e}(\bm{u}^\e_i):\phi_{i,\delta}\bm{e}(\bm{u}^\e_i)d\bm{x}-\int_{D_i}\mc{C}[\bm{e}(\bm{u}),\bm{e}(\bm{u})]d\bm{x} \right|\leq$$
$$\sum_{i=1}^{N_\delta}\left[\limsup_{\e\to 0}C\int_{D_i\setminus D_{i,\delta}}\bm{e}(\bm{u}^\e_i):\bm{e}(\bm{u}^\e_i)d\bm{x}+\limsup_{\e\to 0}\left|\int_{D_i}\left(\mu\bm{e}(\bm{u}^\e_i):\bm{e}(\bm{u}^\e_i)-\mc{C}[\bm{e}(\bm{u}),\bm{e}(\bm{u})]\right)d\bm{x} \right|\right].$$
From the construction of $\bm{u}^\e_i$, $\displaystyle\int_{D_i}\bm{e}(\bm{u}^\e_i):\bm{e}(\bm{u}^\e_i)d\bm{x}$ converges to $\displaystyle\int_{D_i}\mc{C}[\bm{e}(\bm{\chi}_{A_i}),\bm{e}(\bm{\chi}_{A_i})]d\bm{x}$ and 
$$\mc{C}[\bm{e}(\bm{\chi}_{A_i}),\bm{e}(\bm{\chi}_{A_i})]-\mc{C}[\bm{e}(\bm{u}),\bm{e}(\bm{u})]=\mc{C}[\bm{e}(\bm{u})-\bm{e}(\bm{\chi}_{A_i}),\bm{e}(\bm{u})+\bm{e}(\bm{\chi}_{A_i})]\leq C\delta,$$
almost everywhere in $D_i$, so
$$I_{1,\delta}\leq C\sum_{i=1}^{N_\delta}\limsup_{\e\to 0}\int_{D_i\setminus D_{i,\delta}}\bm{e}(\bm{u}^\e_i):\bm{e}(\bm{u}^\e_i)d\bm{x}+C\delta.$$
We obtained for $\bm{u}\in V\cap W^{2,\infty}(D)^d$ and $\delta>0$ the sequence $\bm{u}^\e_\delta$ weakly convergent to $\bm{u}$ in $H_0^1(D)^d$ that satisfies properties (\ref{c1}), i.e.
\begin{equation}\nonumber
\begin{split}
\limsup_{\delta\to 0}\limsup_{\e\to 0}\displaystyle\int_D(\operatorname{div}\bm{u}^\e_\delta)^2 d\bm{x}=\limsup_{\delta\to 0}\limsup_{\e\to 0}\int_{D_r^\e}\bm{e}(\bm{u}^\e_\delta):\bm{e}(\bm{u}^\e_\delta) d\bm{x}\\
=\limsup_{\delta\to 0}\limsup_{\e\to 0}|E^\e_\omega(\bm{u}^\e_\delta)-E^*(u)|=0,\mbox{ a.s. }\omega\in\Omega.
\end{split}
\end{equation}
To this sequence we apply {\bf Lemma \ref{lemmasuf}} for $L=E^*(\bm{u})$ and obtain a sequence $\bm{u}^\e\in V^\e$, weakly convergent to $\bm{u}$ such that $E^\e_\omega(\bm{u}^\e)\to E^*(\bm{u})$ a.s. $\omega\in\Omega$.

We notice that we have proved so far that
$$\Gamma-\limsup_{\e\to 0}E^\e_\omega(\bm{u})\leq E^*(\bm{u}),$$
in the space $V$ equipped with the sequentially weak topology of $H_0^1(D)^d$, for any $\bm{u}$ in a dense subset of $V$ with respect to the strong topology of $H_0^1(D)^d$. We now use the fact that $\Gamma-\displaystyle\limsup_{\e\to 0}E^\e_\omega$ is lower semicontinuous by the definition, and that $E$ is continuous.

Take $\bm{u}\in V$, and let $(\bm{u}_k)_{k\geq 1}$ be a sequence strongly converging to $\bm{u}$ in $H_0^1(D)^d$, such that
$$\Gamma-\limsup_{\e\to 0}E^\e_\omega(\bm{u}_k)\leq E^*(\bm{u}_k).$$
Setting $k\to \infty$ we obtain
$$\Gamma-\limsup_{\e\to 0}E^\e_\omega(\bm{u})\leq\liminf_{k\to\infty}\Gamma-\limsup_{\e\to 0}E^\e_\omega(\bm{u}_k)\leq\liminf_{k\to\infty}E^*(\bm{u}_k),$$
which implies that
$$\Gamma-\limsup_{\e\to 0}E^\e_\omega(\bm{u})\leq E^*(\bm{u}),$$
for every $\bm{u}\in V$.

ii) In the second part we prove 
$$\Gamma-\liminf_{\e\to 0}E^\e_\omega\geq E, \mbox{ a.s. }\omega\in\Omega.$$

For this, we have to show that for every $\bm{u}\in H^1(D)^d$ and for every sequence $\bm{u}^\e$ weakly convergent to $\bm{u}$ we have
$$E^*(\bm{u})\leq \liminf_{\e\to 0} E^\e_\omega(\bm{u}^\e).$$
Take first $\bm{w}$ a smooth vector field, $\bm{w}\in V\cap W^{2,\infty}(D)^d$ and for a small positive $\delta >0$ select a partition $(D_i)_{i=1,N_\delta}$ of $D$, such that for all $1\leq i \leq N_\delta$, $D_i$ is Lipschitz and
$$\mathop{\essosc}_{D_i}\bm{e}(\bm{w})\leq \delta.$$
As before, $\bm{A}_i=\displaystyle\dashint_{D_i}\bm{e}(\bm{w})d\bm{x}$ and let $\bm{w}^\e_i(\bm{x})=\bm{A}_i\bm{x}-\e\bm{\chi}_{\bm{A}_i}\left(\dfrac{\bm{x}}{\e}\right)$ extended by periodicity to the whole $D_i$. Then $\bm{w}^\e_i\rightharpoonup\bm{w}_i$ as $\e\to 0$, where $\bm{w}_i(\bm{x})=\bm{A}_i\bm{x}$. Take also $\phi_i$ a smooth cutoff function supported in $D_i$ with values in between $0$ and $1$. We have:
$$\bm{e}(\bm{u}^\e(\bm{x})):\bm{e}(\bm{u}^\e(\bm{x}))\geq\bm{e}(\bm{w}^\e_i(\bm{x})):\bm{e}(\bm{w}^\e_i(\bm{x}))-2\bm{e}(\bm{w}^\e_i(\bm{x})):\bm{e}(\bm{u}^\e(\bm{x})-\bm{w}^\e_i(\bm{x}))$$
for almost all $\bm{x}\in D_i$. We multiply $\mu\phi_i$ and integrate over $D_i$ and obtain
$$\int_{D_i}\mu\bm{e}(\bm{u}^\e):\bm{e}(\bm{u}^\e)d\bm{x}\geq\int_{D_i}\mu\phi_i\bm{e}(\bm{w}^\e_i):\bm{e}(\bm{w}^\e_i)d\bm{x}-2\int_{D_i}\mu\phi_i\bm{e}(\bm{w}^\e_i):\bm{e}(\bm{u}^\e-\bm{w}^\e_i)d\bm{x}.$$
From the choice of $\bm{w}_i$ we have that
$$\lim_{\e\to 0} \int_{D_i}\mu\phi_i\bm{e}(\bm{w}^\e_i):\bm{e}(\bm{w}^\e_i)d\bm{x}=\int_{D_i}\mu\phi_i\bm{A}_i:\bm{A}_i d\bm{x}+\int_{D_i} \phi_i\mathcal{C}[\bm{A}_i,\bm{A}_i] d\bm{x},$$
thus
\begin{equation}\nonumber
\begin{split}
\liminf_{\e\to 0}\int_{D_i}\mu\bm{e}(\bm{u}^\e):\bm{e}(\bm{u}^\e)d\bm{x}&\geq\int_{D_i}\mu\phi_i\bm{A}_i:\bm{A}_i d\bm{x}+\int_{D_i}\phi_i\mathcal{C}[\bm{A}_i,\bm{A}_i] d\bm{x}\\
&-2\mu\limsup_{\e\to 0}\int_{D_i}\phi_i\bm{e}(\bm{w}^\e_i):\bm{e}(\bm{u}^\e-\bm{w}^\e_i)d\bm{x}
\end{split}
\end{equation}
Assume for now that
$$\limsup_{\e\to 0}\int_{D_i}\phi_i\bm{e}(\bm{w}^\e_i):\bm{e}(\bm{u}^\e-\bm{w}^\e_i)d\bm{x}$$
is bounded from above by $C||\bm{e}(\bm{u}-\bm{w}_i)||_{L^2(D_i)^{d\times d}}$, where the constant $C$ is independent of $D_i$ and $\phi_i$. We now sum over all $1\leq i\leq N_\delta$ which yields
$$\liminf_{\e\to 0} \int_D \mu\bm{e}(\bm{u}^\e):\bm{e}(\bm{u}^\e) d\bm{x}\geq \sum_{i=1}^{N_\delta}\int_{D_i}\left(\mu\phi_i\bm{A}_i:\bm{A}_i + \phi_i\mathcal{C}[\bm{A}_i,\bm{A}_i]\right) d\bm{x}-C||\bm{e}(\bm{u})-\bm{A}_i||_{L^2(D_i)^{d\times d}}.$$
We let now $\phi_i$ to converge increasingly to one in each $D_i$ apply the Lebesgue Beppo-Levy Theorem to obtain
$$ \liminf_{\e\to 0} \int_D \mu\bm{e}(\bm{u}^\e):\bm{e}(\bm{u}^\e) d\bm{x}\geq \sum_{i=1}^{N_\delta}\int_{D_i}\left(\mu\bm{A}_i:\bm{A}_i+ \mathcal{C}[\bm{A}_i,\bm{A}_i]\right) d\bm{x}-C||\bm{e}(\bm{u})-\bm{A}_i||_{L^2(D_i)^{d\times d}}.$$
Set $\delta\to 0$ and the uniform continuity of $\bm{e}(\bm{w})$ in $D$ implies that
$$\lim_{\delta\to 0}\sum_{i=1}^{N_\delta}\left(\int_{D_i}\mu\bm{A}_i:\bm{A}_i d\bm{x}+\int_{D_i} \mathcal{C}[\bm{A}_i,\bm{A}_i] d\bm{x}\right)=\int_D \mu\bm{e}(\bm{w}):\bm{e}(\bm{w}) d\bm{x}+\int_D \mathcal{C}[\bm{e}(\bm{w}),\bm{e}(\bm{w})] d\bm{x},$$
and
$$\lim_{\delta\to 0}\sum_{i=1}^{N_\delta}||\bm{e}(\bm{u})-\bm{A}_i||_{L^2(D_i)^{d\times d}}=||\bm{e}(\bm{u})-\bm{w}||_{L^2(D)^{d\times d}}.$$
Since $\bm{u}^\e\rightharpoonup\bm{u}$ as $\e\to 0$ we have according to the first part from {\bf Theorem \ref{therge}} that a.s. $\omega\in\Omega$
$$\liminf_{\e\to 0} E^\e_\omega(\bm{u}^\e)\geq E^*(\bm{w})-C||\bm{u}-\bm{w}||_{H^1(D)^d}.$$
Now we make $\bm{w}$ converge strongly in $H^1(D)^d$ to $\bm{u}$ and obtain the desired inequality.

It remains to show that
$$\limsup_{\e\to 0}\int_{D_i}\phi_i\bm{e}(\bm{w}^\e_i):\bm{e}(\bm{u}^\e-\bm{w}^\e_i)d\bm{x}\leq C||\bm{e}(\bm{u}-\bm{w}_i)||_{L^2(D_i)^{d\times d}}.$$
For every symmetric matrix $\bm{A}$ with zero trace, and every $\phi\in W^{1,\infty}(D)$ compactly supported in $D$, define the functional $q:V\to \Re\cup\{+\infty\}$ by
$$q(\bm{v})=\sup_{\underset{\bm{v}^\e\rightharpoonup\bm{v}}{\bm{v}^\e\in V^\e}}\limsup_{\e\to 0} \int_{D}\phi\bm{e}(\bm{w}^\e):\bm{e}(\bm{v}^\e)d\bm{x},$$
where $\bm{w}^\e(\bm{x})=\bm{A}\bm{x}-\e\bm{\chi}\left(\dfrac{\bm{x}}{\e}\right)$.
It is clear from the definition that $q$ is sublinear, i.e.
$$q(\lambda\bm{v})=\lambda q(\bm{v}) \ \mbox{ for every positive }\lambda \ \mbox{ and every }\bm{v}\in V,$$
and
$$q(\bm{v}_1+\bm{v}_2)\leq q(\bm{v}_1)+q(\bm{v}_2)\ \mbox{ for every }\bm{v}_1,\bm{v}_2\in V.$$
It is enough to show that $q$ is continuous with respect to the strong topology on $H^1(D)^d$, i.e. there exists a constant $C$ such that for every $\bm{v}\in V$
$$q(\bm{v})\leq C ||\bm{v}||_{H^1(D)^d},$$
and the constant $C$ is proportional with $||\phi||_{L^\infty(D)}$ and with $||\bm{A}||$.

But $q(\bm{v}_1+\bm{v}_2)\leq q(\bm{v}_1)+\displaystyle\limsup_{\e\to 0}\int_{D} \phi\bm{e}(\bm{w}^\e):\bm{e}(\bm{v}_2^\e)d\bm{x}$, for every $\bm{v}_1,\bm{v}_2\in V$, and any weakly convergent sequence $\bm{v}_2^\e$ to $\bm{v}_2$. But from {\bf Lemma \ref{lemmanec}} $q(\bm{0})=0$, and in the first part we showed that for any $\bm{v}\in V$, there exists a sequence $\bm{v}^\e\rightharpoonup\bm{v}$ as $\e\to 0$ such that
$$\limsup_{\e\to 0}||\bm{e}(\bm{v}^\e)||_{L^2(D)^{d\times d}}\leq C ||\bm{v}||_{H^1(D)^d},$$
for a constant independent of $\bm{v}$, which means that for every $\bm{v}\in V$ we have
$$q(\bm{v})\leq \limsup_{\e\to 0} C ||\phi\bm{e}(\bm{w}^\e)||_{L^2(D)^{d\times d}}||\bm{v}||_{H^1(D)^d}\leq C ||\phi||_{L^\infty(D)}||\bm{v}||_{H^1(D)^d}||\bm{A}||.$$
\end{proof}

The following corollary provides us with the existence of the homogenized solution, and the corresponding energy convergence.
\begin{cor}
\label{cor2}
Almost surely $\omega\in\Omega$ the sequence $\bm{u}^\e_\omega\rightharpoonup\bm{u}^*$ in $V$ and $E^\e_\omega(\bm{u}^\e_\omega)\to E^*(\bm{u}^*)$ as $\e\to 0$, where $\bm{u}^*$ is the unique minimizer of the functional $E^*$.
\end{cor}
\begin{proof}
We only need to show that a.s. $\omega\in\Omega$ the functionals $E^\e_\omega$ are equicoercive. After that, from the $\Gamma$-- convergence of functionals $E^\e_\omega$ to $E^*$ and equicoercivity of $E^\e_\omega$ we deduce the convergence of energies as well as the convergence of minimizers (\cite{dalint93} Theorem 7.8 and Corollary 7.17), therefore a.s. $\omega\in\Omega$ we obtain that $\bm{u}^\e_\omega\rightharpoonup\bm{u}^*$ in $V$ and $E^\e(\bm{u}^\e_\omega)\to E^*(\bm{u}^*)$ as $\e\to 0$.

To show equicoercivity we use first (\ref{estB}) and obtain a.s. $\omega\in\Omega$ the following estimate for $\e$ small enough and a constant $C$ independent of $\e$:
$$\left|\sum_{k\in N^\e}\int_{\partial T^\e_k}\e g\left(\frac{\cdot}{\e},\bm{v},\tau_k\omega\right)ds\right|\leq C+\int_D\int_{\partial T}\int_{\Omega}g(s,\bm{v},\omega) dP d\bm{x} ds,$$
for any $\bm{v}\in H^1(D)^d$. Then we use ($\ref{coercE}$) and derive
\begin{equation}\nonumber
\begin{split}
\left|\sum_{k\in N^\e}\int_{\partial T^\e_k}\e g\left(\frac{\cdot}{\e},\bm{v},\tau_k\omega\right)ds\right|&\leq C+\left|\int_D\int_{\partial T}\int_{\Omega}g(s,\bm{0},\omega) dP d\bm{x} ds\right|+C||\bm{v}||_{L^2(D)^d}^\gamma\\
&\leq C +C|\bm{v}|^\gamma,
\end{split}
\end{equation}
for a constant $C$ independent of $\e$. Finally, the definition (\ref{defE^e_omega}) of $E^\e_\omega$ gives us that a.s. $\omega\in\Omega$, there exist two constants $A>0$ and $C$, both independent of $\e$, such that for all $\bm{v}\in H^1(D)^d$ and for $\e$ small enough we have 
$$E^\e_\omega(\bm{v})\geq -C - C||\bm{v}||_{L^2(D)^d}^\gamma-C||\bm{v}||_{L^2(D)^d} +A||\bm{v}||_{L^2(D)^d}^2, $$
which implies the equicoercivity of $E^\e_\omega$.
\end{proof}
\subsection{Corrector Result}
The $\Gamma$-- convergence result provided by {\bf Theorem \ref{thGammaconv}} together with the {\bf Corollary \ref{cor2}} gives us only weak convergence in $H^1(D)^d$ of the solution $\bm{u}^\e_\omega$ to the homogenized solution $\bm{u}^*$. Next, we formulate the first order corrector result to improve the weak convergence 
to strong convergence.

First, we define for every $1\leq i,j\leq d$ the matrix $\bm{E}_{ij}\in\mc{S}$ componentwise by
\begin{equation}
\label{defE_ij}
(E_{ij})_{kl}=\frac{1}{2}\left(\delta_{ik} \delta_{jl} + \delta_{il} \delta_{jk}\right),
\end{equation}
for every $1\leq k,l\leq d$ and redenote for simplicity $\bm{\chi}_{ij}:=\bm{\chi}_{\bm{E}_{ij}}$. Since $\bm{\chi}_{ij}\in\bm{K}_{\bm{E}_{ij}}$ we use the definition (\ref{defK_A}) of the sets $\bm{K}_{\bm{A}}$, and we obtain that $\bm{e}(\bm{\chi}_{\bm{E}_{ij}})=\bm{E}_{ij}$ in $T$. The definition (\ref{defE_ij}) will imply that
\begin{equation}
\label{p1}
\operatorname{div}\bm{\chi}_{ij}=\delta_{ij} \mathds{1}_T.
\end{equation}

\begin{theorem}
\label{thcorr}
Assume that the homogenized velocity $\bm{u}^*\in V\cap W^{2,\infty}(D)^d$, then a.s. $\omega\in\Omega$, the sequence $\bm{u}^\e_\omega-\left(\bm{u}^*-\displaystyle\sum_{i,j=1}^d\e\bm{\chi}_{ij}\left(\frac{\cdot}{\e}\right)e_{ij}(\bm{u}^*)\right)\to\bm{0}$ in $H^1(D)^d$ as $\e\to 0$, where $\bm{\chi}_{ij}$ is 
extended $Y$-- periodically to the whole $D$.
\end{theorem}
\begin{proof}
We notice that we need the assumption that $\bm{u}^*\in V\cap W^{2,\infty}(D)^d$ in order for $\bm{r}^\e=\displaystyle\sum_{i,j=1}^d\e\bm{\chi}_{ij}\left(\frac{\cdot}{\e}\right) e_{ij}(\bm{u}^*)\in H^1(D)^d$. First, we see that
$$||\bm{u}^\e_\omega-(\bm{u}^*-\bm{r}^\e)||_{L^2(D)^d}\leq||\bm{u}^\e_\omega-\bm{u}^*||_{L^2(D)^d}+\e\sum_{i,j=1}^d ||\bm{\chi}_{ij}\left(\frac{\cdot}{\e}\right)e_{ij}(\bm{u}^*)||_{L^2(D)^d}$$
and because $\bm{u}^*\in W^{2,\infty}(D)^d$ we have
$$||\bm{u}^\e_\omega-(\bm{u}^*-\bm{r}^\e)||_{L^2(D)^d}\leq||\bm{u}^\e_\omega-\bm{u}^*||_{L^2(D)^d}+\e||\bm{u}^*||_{W^{2,\infty}(D)^{d}}\sum_{i,j=1}^d ||\bm{\chi}_{ij}\left(\frac{\cdot}{\e}\right)||_{L^2(D)^d}$$
which implies that $\bm{u}^\e_\omega-(\bm{u}^*-\bm{r}^\e)\to\bm{0}$ in $L^2(D)^d$ as $\e\to 0$. If we are able to show that $\bm{e}(\bm{u}^\e_\omega-(\bm{u}^*-\bm{r}^\e))\to\bm{0}$ in $L^2(D)^{d\times d}$ as $\e\to 0$, then applying Korn's inequality we obtain the desired result.

To show this we use results of {\bf Theorem \ref{thGammaconv}}. An elementary calculation gives
$$\bm{e}(\bm{u}^\e_\omega-\bm{u}^*+\bm{r}^\e):\bm{e}(\bm{u}^\e_\omega-\bm{u}^*+\bm{r}^\e) =\bm{e}(\bm{u}^\e_\omega):\bm{e}(\bm{u}^\e_\omega)+\bm{e}(\bm{u}^*-\bm{r}^\e):\bm{e}(\bm{u}^*-\bm{r}^\e)-2\bm{e}(\bm{u}^\e_\omega):\bm{e}(\bm{u}^*-\bm{r}^\e).$$
The variational formulation (\ref{eqvarforhe}) yields 
$$\displaystyle\int_D 2\mu\bm{e}(\bm{u}^\e_\omega):\bm{e}(\bm{u}^\e_\omega)d\bm{x}+\sum_{k\in N^\e}\int_{\partial T^\e_k}\e\bm{\nabla}_{\bm{u}} g\left(\dfrac{\cdot}{\e},\bm{u}^\e_\omega,\tau_k\omega\right) \cdot\bm{u}^\e_\omega ds=\int_D\bm{f}^\e\cdot\bm{u}^\e_\omega d\bm{x},$$
then, using the second part of {\bf Theorem \ref{therge}}
\begin{equation}\nonumber
\begin{split}
\lim_{\e\to 0}\displaystyle\int_D 2\mu\bm{e}(\bm{u}^\e_\omega):\bm{e}(\bm{u}^\e_\omega)d\bm{x} &=\int_D\bm{f}\cdot\bm{u}^* d\bm{x}-\int_D\int_{\partial T}\int_{\Omega}\bm{\nabla}_{\bm{u}} g\left(s,\bm{u}^*,\omega\right) \cdot\bm{u}^* dPdsd\bm{x}\\
&=\int_D\bm{f}\cdot\bm{u}^* d\bm{x}+\int_D\bm{f}^*\left(\bm{u}^*\right) \cdot\bm{u}^* d\bm{x}
\end{split}
\end{equation}
Also, because $\bm{r}^\e\rightharpoonup\bm{0}$ in $H^1(D)^d$ as $\e\to 0$,
$$\displaystyle\lim_{\e\to 0}\int_D 2\mu\bm{e}(\bm{u}^*-\bm{r}^\e):\bm{e}(\bm{u}^*-\bm{r}^\e)d\bm{x}=\int_D 2\mu\bm{e}(\bm{u}^*):\bm{e}(\bm{u}^*)d\bm{x}+\lim_{\e\to 0}\int_D 2\mu\bm{e}(\bm{r}^\e):\bm{e}(\bm{r}^\e)d\bm{x}$$
But
\begin{equation}\nonumber
\begin{split}
\lim_{\e\to 0}&\int_D 2\mu\bm{e}\left(\e\bm{\chi}_{ij}\left(\frac{\cdot}{\e}\right) e_{ij}(\bm{u}^*)\right):\bm{e}\left(\e\bm{\chi}_{kl}\left(\frac{\cdot}{\e}\right) e_{kl}(\bm{u}^*)\right) d\bm{x}=\\
\lim_{\e\to 0}&\int_D 2\mu\bm{e}(\bm{\chi}_{ij})\left(\frac{\cdot}{\e}\right) :\bm{e}(\bm{\chi}_{kl})\left(\frac{\cdot}{\e}\right) e_{ij}(\bm{u}^*)e_{kl}(\bm{u}^*) d\bm{x}=\int_D 2\mathcal{C} [\bm{E}_{ij},\bm{E}_{kl}] e_{ij}(\bm{u}^*)e_{kl}(\bm{u}^*) d\bm{x},
\end{split}
\end{equation}
hence
$$\lim_{\e\to 0}\int_D 2\mu\bm{e}(\bm{r}^\e):\bm{e}(\bm{r}^\e) d\bm{x}= \sum_{i,j=1}^d\int_D 2\mathcal{C} [\bm{E}_{ij},\bm{E}_{kl}] e_{ij}(\bm{u}^*)e_{kl}(\bm{u}^*) d\bm{x}=\int_D 2\mathcal{C}[\bm{e}(\bm{u}^*),\bm{e}(\bm{u}^*)] d\bm{x}$$
which implies that
\begin{equation}\nonumber
\begin{split}
\displaystyle\lim_{\e\to 0}\int_D 2\mu\bm{e}(\bm{u}^*-\bm{r}^\e):\bm{e}(\bm{u}^*-\bm{r}^\e)d\bm{x}&=\int_D 2\mu\bm{e}(\bm{u}^*):\bm{e}(\bm{u}^*)d\bm{x}+\int_D 2\mathcal{C}[\bm{e}(\bm{u}^*),\bm{e}(\bm{u}^*)] d\bm{x}\\
&=\int_D\bm{f}\cdot\bm{u}^* d\bm{x}+\int_D\bm{f}^*(\bm{u}^*)\cdot\bm{u}^* d\bm{x},
\end{split}
\end{equation}
after we use the variational formulation (\ref{eqvarforh2}) for $\bm{u}^*$.
Let us show now
$$\displaystyle\lim_{\e\to 0}\int_D 2\mu\bm{e}(\bm{u}^\e_\omega):\bm{e}(\bm{u}^*-\bm{r}^\e)d\bm{x}=\displaystyle\int_D\bm{f}\cdot\bm{u}^* d\bm{x}+\int_D\bm{f}^*(\bm{u}^*)\cdot\bm{u}^* d\bm{x}.$$
This is equivalent to showing that
$$\limsup_{\delta\to 0}\limsup_{\e\to 0}\left|\int_D 2\mu\bm{e}(\bm{u}^\e_\omega):\bm{e}(\bm{u}^*-\bm{r}^\e)d\bm{x}-\int_D\bm{f}\cdot\bm{u}^* d\bm{x}-\int_D\bm{f}^*(\bm{u}^*)\cdot\bm{u}^* d\bm{x}\right|=0$$
because the sequence $\left(\bm{u}^\e_\omega\right)$ does not depend on $\delta$. For a fixed $\delta>0$ we denote by $D_\delta$ the set of points from $D$ with the distance to the boundary of $D$ at least $\delta$ and select a smooth cutoff function $\phi_\delta$ supported in $D$, such that $0\leq\phi_\delta\leq 1$ in $D$, equal to 1 in $D_\delta$ and such that $||\bm{\nabla}\phi_\delta||_{L^\infty(D)^d}\leq \dfrac{C}{\delta}$ and see that
$$\int_D 2\mu\bm{e}(\bm{u}^\e_\omega):\bm{e}(\bm{u}^*-\bm{r}^\e)d\bm{x}=\int_D 2\mu\bm{e}(\bm{u}^\e_\omega):\bm{e}(\bm{u}^*-\bm{r}^\e)\phi_\delta d\bm{x}+\int_D 2\mu\bm{e}(\bm{u}^\e_\omega):\bm{e}(\bm{u}^*-\bm{r}^\e) (1-\phi_\delta)d\bm{x}.$$
We estimate the second term, and a.s. $\omega\in\Omega$
$$\left|\int_D 2\mu\bm{e}(\bm{u}^\e_\omega):\bm{e}(\bm{u}^*-\bm{r}^\e) (1-\phi_\delta)d\bm{x}\right|\leq C||\bm{u}^\e_\omega||_{H^1(D)^d}||\bm{u}^*-\bm{r}^\e||_{H^1(D\setminus D_\delta)^d}$$
$$\leq C \left(||\bm{u}^*||_{H^1(D\setminus D_\delta)^d}+\sum_{i,j=1}^{d}||\e\bm{\chi}_{ij}\left(\frac{\cdot}{\e}\right)||_{H^1(D\setminus D_\delta)^d}\right)$$
which implies that
$$\limsup_{\e\to 0}\left|\int_D 2\mu\bm{e}(\bm{u}^\e_\omega):\bm{e}(\bm{u}^*-\bm{r}^\e) (1-\phi_\delta)d\bm{x}\right|\leq C\left(||\bm{u}^*||_{H^1(D\setminus D_\delta)^d}+\sum_{i,j=1}^d \int_{D\setminus D_\delta} \mathcal{C}[\bm{E}_{ij},\bm{E}_{ij}] d\bm{x}\right),$$
Therefore
$$\limsup_{\delta\to 0}\limsup_{\e\to 0} \left|\int_D 2\mu\bm{e}(\bm{u}^\e_\omega):\bm{e}(\bm{u}^*-\bm{r}^\e) (1-\phi_\delta)d\bm{x}\right|=0.$$
It is sufficient to show that
$$\limsup_{\delta\to 0}\limsup_{\e\to 0}\left|\int_D 2\mu\bm{e}(\bm{u}^\e_\omega):\bm{e}(\bm{u}^*-\bm{r}^\e)\phi_\delta d\bm{x}-\int_D\bm{f}\cdot\bm{u}^* d\bm{x}-\int_D\bm{f}^*(\bm{u}^*)\cdot\bm{u}^* d\bm{x}\right|=0.$$
We perform the following calculation:
\begin{equation}\nonumber
\begin{split}
&\int_D 2\mu\bm{e}(\bm{u}^\e_\omega):\bm{e}(\bm{u}^*-\bm{r}^\e)\phi_\delta d\bm{x}=\int_D 2\mu\bm{e}(\bm{u}^\e_\omega):\bm{e}(\bm{u}^*-\bm{r}^\e\phi_\delta) d\bm{x}+\\
&\int_D 2\mu\bm{e}(\bm{u}^\e_\omega):\bm{e}(\bm{u}^*)(\phi_\delta-1) d\bm{x}+\int_D 2\mu \sum_{i,j=1}^d e_{ij}(\bm{u}^\e_\omega)\left(\frac{\partial \phi_\delta}{\partial x_i}\bm{r}^\e_j+\frac{\partial \phi_\delta}{\partial x_j}\bm{r}^\e_i\right) d\bm{x},
\end{split}
\end{equation}
and estimate the second term
$$\limsup_{\e\to 0}\left|\int_D 2\mu\bm{e}(\bm{u}^\e_\omega):\bm{e}(\bm{u}^*)(\phi_\delta-1)\right|\leq C||\bm{u}^*||_{H^1(D\setminus D_\delta)^d}.$$
Because $\bm{r}^\e\to 0$ in $L^2(D)^d$ as $\e\to 0$
$$\limsup_{\e\to 0}\left|\int_D 2\mu \sum_{i,j=1}^d
e_{ij}(\bm{u}^\e_\omega)\left(\frac{\partial \phi_\delta}{\partial x_i}\bm{r}^\e_j+\frac{\partial \phi_\delta}{\partial x_j}\bm{r}^\e_i\right)\right|=0.$$
It remains to show that
$$\limsup_{\delta\to 0}\limsup_{\e\to 0}\left|\int_D 2\mu\bm{e}(\bm{u}^\e_\omega):\bm{e}(\bm{u}^*-\bm{r}^\e\phi_\delta) d\bm{x}-\int_D\bm{f}\cdot\bm{u}^* d\bm{x}-\int_D\bm{f}^*(\bm{u}^*)\cdot\bm{u}^* d\bm{x}\right|=0.$$
For the sequence $\bm{u}^*-\bm{r}^\e\phi_\delta$ we calculate $||\operatorname{div}(\bm{u}^*-\bm{r}^\e\phi_\delta)||_{L^2(D)}$ and $||\bm{e}(\bm{u}^*-\bm{r}^\e\phi_\delta)||_{L^2(D^\e_r)^{d\times d}}$.
\begin{equation}\nonumber
\begin{split}
\operatorname{div}(\bm{u}^*-\bm{r}^\e\phi_\delta)&=-\phi_\delta \operatorname{div}\bm{r}^\e-\bm{r}^\e\cdot\bm{\nabla} \phi_\delta\\
&=-\phi_\delta \e\sum_{i,j=1}^d\bm{\chi}_{ij}\left(\frac{\cdot}{\e}\right)\cdot\bm{\nabla} e_{ij}(\bm{u}^*)-\phi_\delta \sum_{i,j=1}^d \operatorname{div}\bm{\chi}_{ij}\left(\frac{\cdot}{\e}\right) e_{ij}(\bm{u}^*)-\bm{r}^\e\cdot\bm{\nabla} \phi_\delta.
\end{split}
\end{equation}
According to the property (\ref{p1}) in any set of the form $\e k+\e T$, for $k\in \Int^d$
$$\displaystyle\sum_{i,j=1}^d \operatorname{div}\bm{\chi}_{ij}\left(\frac{\cdot}{\e}\right) e_{ij}(\bm{u}^*)=\displaystyle\sum_{i,j=1}^d \delta_{ij} e_{ij}(\bm{u}^*)=\operatorname{div}\bm{u}^*=0,$$
and because $\operatorname{div}\bm{\chi}_{ij}=0$ outside the sets $\e k+\e T$, we obtain that $\displaystyle\sum_{i,j=1}^d \operatorname{div}\bm{\chi}_{ij}\left(\frac{\cdot}{\e}\right) e_{ij}(\bm{u}^*)=0$ in $D$. Thus 
$$\lim_{\e\to 0} ||\operatorname{div}(\bm{u}^*-\bm{r}^\e\phi_\delta)||_{L^2(D)}=0.$$
Also in $D_r^\e$
\begin{equation}\nonumber
\begin{split}
\bm{e}(\bm{u}^*-\bm{r}^\e\phi_\delta)&:\bm{e}(\bm{u}^*-\bm{r}^\e\phi_\delta)=\bm{e}(\bm{u}^*-\bm{r}^\e):\bm{e}(\bm{u}^*-\bm{r}^\e)+\bm{e}(\bm{r}^\e(1-\phi_\delta)):\bm{e}(\bm{r}^\e(1-\phi_\delta))\\
&\leq	\bm{e}(\bm{u}^*-\bm{r}^\e):\bm{e}(\bm{u}^*-\bm{r}^\e)+ C\bm{e}(\bm{r}^\e):\bm{e}(\bm{r}^\e)(1-\phi_\delta)^2+C|\bm{r}^\e|^2|\bm{\nabla}\phi_\delta|^2\\
&\leq	C\e^2+C\bm{e}(\bm{r}^\e):\bm{e}(\bm{r}^\e)(1-\phi_\delta)^2+C|\bm{r}^\e|^2|\bm{\nabla}\phi_\delta|^2.
\end{split}
\end{equation}
Hence,
$$\limsup_{\delta\to 0}\limsup_{\e\to 0}\int_{D_r^\e}\bm{e}(\bm{u}^*-\bm{r}^\e\phi_\delta):\bm{e}(\bm{u}^*-\bm{r}^\e\phi_\delta) d\bm{x}\leq \limsup_{\delta\to 0}\int_{D\setminus D_\delta}\bm{e}(\bm{r}^\e):\bm{e}(\bm{r}^\e) d\bm{x}=0.$$
We apply {\bf Lemma \ref{lemmaconstruction}} to the sequence $(\bm{u}^*-\bm{r}^\e\phi_\delta)$ and find a sequence $\bm{u}^\e_\delta \in V^\e$ such that
$$\limsup_{\delta\to 0}\limsup_{\e\to 0}||\bm{u}^*-\bm{r}^\e\phi_\delta-\bm{u}^\e_\delta||_{H^1(D)^{d}}=0.$$
This means that
$$\limsup_{\delta\to 0}\limsup_{\e\to 0}\left|\int_D 2\mu\bm{e}(\bm{u}^\e_\omega):\bm{e}(\bm{u}^*-\bm{r}^\e\phi_\delta) d\bm{x}-\int_D\bm{f}\cdot\bm{u}^* d\bm{x}-\int_D\bm{f}^*(\bm{u}^*)\cdot\bm{u}^* d\bm{x}\right|=$$
$$\limsup_{\delta\to 0}\limsup_{\e\to 0}\left|\int_D 2\mu\bm{e}(\bm{u}^\e_\omega):\bm{e}(\bm{u}^\e_\delta) d\bm{x}-\int_D\bm{f}\cdot\bm{u}^* d\bm{x}-\int_D\bm{f}^*(\bm{u}^*)\cdot\bm{u}^* d\bm{x}\right|\leq$$
$$\limsup_{\delta\to 0}\limsup_{\e\to 0}\left|\int_D\bm{f}^\e\cdot\bm{u}^\e_\delta d\bm{x}-\int_D\bm{f}\cdot\bm{u}^* d\bm{x}\right|+$$
$$\limsup_{\delta\to 0}\limsup_{\e\to 0}\left|-\sum_{k\in N^\e}\int_{\partial T^\e_k}\bm{\nabla}_{\bm{u}} g\left(\frac{\cdot}{\e},\bm{u}^\e_\delta,\tau_k\omega\right)\cdot\bm{u}^\e_\delta ds-\int_D\bm{f}^*(\bm{u}^*)\cdot\bm{u}^* d\bm{x}\right|=0,$$
after application of the second part of {\bf Theorem \ref{therge}}.
\end{proof}

\section{Conclusions}
\label{sec5}
The current paper provides a homogenization-type result stated in {\bf Theorem \ref{thGammaconv}} and its corollary for a suspension of rigid particles in the presence of highly oscillatory random interfacial forces. In particular, it states that the solution $\bm{u}^\e_\omega$ to the family of problems (\ref{systemhe}) corresponding to a size of the microstructure $\e$ weakly converges in $H^1(D)^d$ as $\e \to 0$ to the solution $\bm{u}^*$ of the homogenized problem (\ref{systemh}). This $\bm{u}^*$ is a unique minimizer of a $\Gamma$-- limit for the sequence of functionals $(E^\e_{\omega})_{\e>0}$ corresponding to the original family of the problems. It is also shown in {\bf Theorem \ref{thcorr}} that a first order corrector of $\bm{u}^*$ constructed based on the solutions of the cell problems yields a strong $H^1(D)^d$-- convergence.

To indicate possible future directions of generalization of the considered case we mention the
study of different types of surface forces with coupled electrostatic or magnetic
phenomena added. Another important and challenging task is to develop computational techniques to analyze the
relative importance of electrostatic interactions and surface energy in materials with ionized
inclusions. Ultimately, such a fundamental understanding of particulate flows
may be used for the control and optimization of the response in the case of the design of novel materials.

\section{Appendix}

Here we prove {\bf Theorems \ref{thvarfor}, \ref{thminfor}} and {\bf \ref{thexun}}.

\subsection{Proof of Theorem \ref{thvarfor}}

\begin{proof}
We notice first that according to {\bf Remark \ref{remark3}}, the integral $\displaystyle\int_{D_r^l}\bm{f}_l(\bm{x},\bm{u}) \cdot\bm{\phi}ds$ makes sense if $\bm{u}$ belongs to $H_0^1(D)^d$. Let us assume that $\{\bm{u},p\}$ is a weak
solution for the system (\ref{system}), so $-\operatorname{div}\bm{\sigma} =\bm{f}$ weakly in $D_f$. This means that $$\int_{D_f} 2\mu\bm{e}(\bm{u}):\bm{e}(\bm{\phi}) d\bm{x}-
\int_{D_f} p\operatorname{div}\bm{\phi}= \int_{D_f}\bm{f}\cdot\bm{\phi}d\bm{x},$$
for every $\bm{\phi}$ from $H_0^1(D_f)^d$. Now, if we take a $\bm{\phi}$ from $H_0^1(D)^d$, such that $\bm{e}(\bm{\phi})=\bm{0}$ in $D_r$, we obtain
$$\int_{D_f} 2\mu\bm{e}(\bm{u}):\bm{e}(\bm{\phi}) d\bm{x}-\int_{D_f} p\operatorname{div}\bm{\phi}= \int_{D_f}\bm{f}\cdot\bm{\phi}d\bm{x}+
\int_{\partial D_f}\bm{\sigma}\bm{n}\cdot\bm{\phi} ds.$$
But
$$\int_{\partial D_f}\bm{\sigma}\bm{n}\cdot\bm{\phi} ds=\int_{\partial D_r}\bm{\sigma}\bm{n}\cdot\bm{\phi} ds=
\sum_{l=1}^N\int_{\partial D_r^l}\bm{\sigma}\bm{n}\cdot\bm{\phi} ds,$$
by our convention that $\bm{n}$ represents the unit normal pointing outside the fluid region. In each $D_r^l$ the vector field $\bm{\phi}$, having the property that $\bm{e}(\bm{\phi})=\bm{0}$, takes the form
$$\bm{\phi}=\bm{M_l} (\bm{x}-\bm{x}_l)+\bm{m}_l,$$
where $\bm{M_l}\in \Re^{d\times d}$ is a skew symmetric matrix and $\bm{m}_l\in \Re^d$ is a fixed vector. Thus
$$\int_{\partial D_r^l}\bm{\sigma}\bm{n}\cdot\bm{\phi} ds=\int_{\partial D_r^l}\bm{\sigma}\bm{n}\cdot (\bm{M_l} (\bm{x}-\bm{x}_l)+\bm{m}_l) ds=$$
$$\int_{\partial D_r^l} \sum_{i,j=1}^d(\bm{\sigma}\bm{n})_i (\bm{M_l})_{ij} (\bm{x}-\bm{x}_l)_jds+\int_{\partial D_r^l}\bm{\sigma}\bm{n}\cdot\bm{m}_l ds=$$
$$\int_{\partial D_r^l} \sum_{i< j}(\bm{M_l})_{ij} \left((\bm{\sigma}\bm{n})_i (\bm{x}-\bm{x}_l)_j-(\bm{\sigma}\bm{n})_j (\bm{x}-\bm{x}_l)_i\right)ds+\int_{\partial D_r^l}\bm{\sigma}\bm{n}\cdot\bm{m}_l ds=$$
$$\sum_{i< j}(\bm{M_l})_{ij}\int_{\partial D_r^l} (\bm{\sigma}\bm{n})_i (\bm{x}-\bm{x}_l)_j-(\bm{\sigma}\bm{n})_j (\bm{x}-\bm{x}_l)_ids+\bm{m}_l\cdot\int_{\partial D_r^l}\bm{\sigma}\bm{n} ds.$$
Using now the balance of forces (\ref{balance.f}) we have that 
$$\int_{\partial D_r^l}\bm{\sigma}\bm{n} ds=\int_{D_r^l}\bm{f} d\bm{x} + \int_{\partial D_r^l}\bm{f}_l(\bm{x},\bm{u})ds,$$
and using the balance of torques (\ref{balance.t}) we have that
\begin{equation}
\nonumber
\begin{split}
&\int_{\partial D_r^l}\left((\bm{\sigma}\bm{n})_i (\bm{x}-\bm{x}_l)_j-(\bm{\sigma}\bm{n})_j (\bm{x}-\bm{x}_l)_i\right)ds=\\
&\int_{D_r^l}\left(\bm{f}_i(\bm{x}-\bm{x}_l)_j-\bm{f}_j (\bm{x}-\bm{x}_l)_i\right)d\bm{x}+\int_{\partial D_r^l}\left((\bm{f}_l)_i(\bm{x},\bm{u}) (\bm{x}-\bm{x}_l)_j- (\bm{f}_l)_j(\bm{x},\bm{u}) (\bm{x}-\bm{x}_l)_i\right)ds.
\end{split}
\end{equation}
Combining the last three equations we obtain
\begin{equation}
\nonumber
\begin{split}
\int_{\partial D_r^l}\bm{\sigma}\bm{n}\cdot\bm{\phi} ds&=\int_{D_r^l} \sum_{i,j=1}^d\bm{f}_i (\bm{M_l})_{ij} (\bm{x}-\bm{x}_l)_jd\bm{x}+\int_{D_r^l}\bm{f}\cdot\bm{m}_l d\bm{x}+\\
&+\int_{\partial D_r^l} \sum_{i,j=1}^d(\bm{f}_l)_i (\bm{M_l})_{ij} (\bm{x}-\bm{x}_l)_jds+\int_{\partial D_r^l}\bm{f}_l\cdot\bm{m}_l ds\\
&=\int_{D_r^l}\bm{f}\cdot\bm{\phi} d\bm{x}+\int_{\partial D_r^l}\bm{f}_l\cdot\bm{\phi} ds.
\end{split}
\end{equation}
Adding over $0\leq l\leq N$ we get
$$\int_{\partial D_r}\bm{\sigma}\bm{n}\cdot\bm{\phi} ds=\int_{D_r}\bm{f}\cdot\bm{\phi} d\bm{x}+\sum_{l=1}^N\int_{\partial D_r^l}\bm{f}_l\cdot\bm{\phi} ds$$
so the variational formulation
$$\int_{D_f} 2\mu\bm{e}(\bm{u}):\bm{e}(\bm{\phi}) d\bm{x}-\int_{D_f} p\operatorname{div}\bm{\phi}= \int_{D_f}\bm{f}\cdot\bm{\phi}d\bm{x}+\int_{\partial D_f}\bm{\sigma}\bm{n}\cdot\bm{\phi} ds$$
becomes, because $\bm{e}(\bm{\phi})=\bm{0}$ in $D_r$
$$\int_{D} 2\mu\bm{e}(\bm{u}):\bm{e}(\bm{\phi}) d\bm{x}-\int_{D_f} p\operatorname{div}\bm{\phi}= \int_{D}\bm{f}\cdot\bm{\phi}d\bm{x}+\sum_{l=1}^N\int_{\partial D_r^l}\bm{f}_l\cdot\bm{\phi} ds,$$
which is what we wanted to prove.

Conversely, assuming (\ref{eqvarfor}) we obtain for all $\bm{\phi}$ from $H_0^1(D_f)^d$
$$\int_{D_f} 2\mu\bm{e}(\bm{u}):\bm{e}(\bm{\phi}) d\bm{x}-\int_{D_f} p\operatorname{div}\bm{\phi}= \int_{D_f}\bm{f}\cdot\bm{\phi}d\bm{x},$$
which gives $-\operatorname{div}\bm{\sigma}=\bm{f}$ in $D_f$. Using this equation, and taking $\bm{\phi}$ from $H_0^1(D)^d$ with $\bm{e}(\bm{\phi})=\bm{0}$ in $D_r$ we obtain that
$$\sum_{l=1}^{N}\int_{\partial D_r^l}\bm{f}_l(\bm{x},\bm{u}) \cdot\bm{\phi}ds=\sum_{l=1}^{N}\int_{\partial D_r^l}\bm{\sigma}\bm{n}\cdot\bm{\phi}ds,$$
which gives us in the same way as in first part the balances of forces and torques.
\end{proof}

\subsection{Proof of Theorem \ref{thminfor}}

\begin{proof}
We notice first that according to {\bf Remark \ref{remark3}} the integral $\displaystyle\int_{\partial D_r^l}g_l(\bm{x},\bm{v})ds$ makes sense for $\bm{v}$ in $H_0^1(D)^d$. From {\bf Theorem \ref{thvarfor}}, if $\{\bm{u},p\}$ is a weak solution for the system (\ref{system}), then
\begin{equation}
\label{eqvarfor2}
\int_D 2\mu\bm{e}(\bm{u}):\bm{e}(\bm{\phi})d\bm{x}=\int_D\bm{f}\cdot\bm{\phi}d\bm{x}-\sum_{l=1}^{N}\int_{\partial D_r^l}\bm{\nabla}_{\bm{u}}g_l(\bm{x},\bm{u}) \cdot\bm{\phi}ds,
\end{equation}
for all $\bm{\phi}\in V_r$. So
$$E_r(\bm{u}+\bm{\phi})-E_r(\bm{u})=\int_D \mu\bm{e}(\bm{\phi}):\bm{e}(\bm{\phi})d\bm{x}+\sum_{l=1}^{N}\int_{\partial D_r^l}\left(g_l(\bm{x},\bm{u}+\bm{\phi})-g_l(\bm{x},\bm{u})+\bm{\nabla}_{\bm{u}}g_l(\bm{x},\bm{u}) \cdot\bm{\phi}\right)ds,$$
which is positive, based on the convexity property of $g_l$ (see {\bf Remark \ref{remark3}}). The quadratic term $\displaystyle\int_D \mu\bm{e}(\bm{v}):\bm{e}(\bm{v})d\bm{x}$ makes th functional $E_r$ strictly convex, therefore the minimizer is unique.
\end{proof}

\subsection{Proof of Theorem \ref{thexun}}

\begin{proof}
Let $\bm{u}$ be the unique minimizer of $E_r$, and let us show that $\bm{u}$ satisfies (\ref{eqvarfor2}). For every $\lambda\in\Re$ and every $\bm{\phi}\in V_r$ we have that
$$E_r(\bm{u}+\lambda\bm{\phi})-E_r(\bm{u})\geq 0,$$
which gives after calculations that
$$\lim
_{\lambda\searrow 0}\int_D 2\mu\bm{e}(\bm{u}):\bm{e}(\bm{\phi})d\bm{x}-\int_D\bm{f}\bm{\phi}d\bm{x}+\sum_{l=1}^{N}\int_{\partial D_r^l}\frac{g_l(\bm{x},\bm{u}+\lambda\bm{\phi})-g_l(\bm{x},\bm{u})}{\lambda}ds\geq 0,$$
and
$$\lim
_{\lambda\nearrow 0}\int_D 2\mu\bm{e}(\bm{u}):\bm{e}(\bm{\phi})d\bm{x}-\int_D\bm{f}\bm{\phi}d\bm{x}+\sum_{l=1}^{N}\int_{\partial D_r^l}\frac{g_l(\bm{x},\bm{u}+\lambda\bm{\phi})-g_l(\bm{x},\bm{u})}{\lambda}ds\leq 0.$$
Combining these two we obtain (\ref{eqvarfor2}), which is
$$\int_D 2\mu\bm{e}(\bm{u}):\bm{e}(\bm{\phi})d\bm{x}=\int_D\bm{f}\cdot\bm{\phi}d\bm{x}-\sum_{l=1}^{N}\int_{\partial D_r^l}\bm{\nabla}_{\bm{u}}g_l(\bm{x},\bm{u}) \cdot\bm{\phi}ds,$$
for all $\bm{\phi}\in V_r$. If $\bm{\phi}$ is a divergence free vector field from $H_0^1(D_f)^d$ we obtain 
$$\int_D 2\mu\bm{e}(\bm{u}):\bm{e}(\bm{\phi})d\bm{x}=\int_D\bm{f}\cdot\bm{\phi}d\bm{x},$$
and from here the existence of $p$, unique up to an additive constant, such that $-\operatorname{div}\bm{\sigma}=\bm{f}$ weakly in $D_f$ (see \cite{temam}). As in {\bf Theorem \ref{thvarfor}} we obtain
$$\sum_{l=1}^{N}\int_{\partial D_r^l}-\bm{\nabla}_{\bm{u}}g_l(\bm{x},\bm{u}) \cdot\bm{\phi}ds=\sum_{l=1}^{N}\int_{\partial D_r^l}\bm{\sigma}\bm{n}\cdot\bm{\phi}ds,$$
which gives us the balances of forces and torques.
\end{proof}

\vspace{5pt}


\end{document}